\documentclass[10pt]{amsart}
\usepackage{graphics}
\usepackage[english]{babel}
\usepackage{amssymb,amsmath,amsfonts}
\usepackage{mathabx}
\usepackage{datetime}
\usepackage{color}
\usepackage{ulem}
\usepackage{lipsum}
\usepackage{amsfonts}
\usepackage{enumerate}
\usepackage{cite}
\usepackage{tikz}
\usetikzlibrary{matrix}
\RequirePackage{amscd}
\RequirePackage{epic}
\RequirePackage[all]{xy}
\RequirePackage{url}
\RequirePackage[shortlabels]{enumitem}
\usepackage{empheq}
\usepackage{epsfig}
 \usepackage{perpage}
 \MakePerPage{linenumbers}
\usepackage[english]{babel}

\usepackage[utf8]{inputenc}
\usepackage{cite}
\usepackage{empheq}
\usepackage{nomencl}
\usepackage{epsfig}
\usepackage{graphicx}
\usepackage{marginnote}

\newcommand{\comment}[1]{}
   
    \newcommand{\set}[1]{{\left\{#1\right\}}}

\newcommand{\T}{\mathbb{T}}
\newcommand{\Z}{\mathbb{Z}}
\newcommand{\R}{\mathbb{R}}
\newcommand{\C}{\mathbb{C}}

\newcommand{\eps}{\varepsilon}

\newcommand{\co}[1]{\textit{#1}}

\usepackage{amsthm}

\RequirePackage{url}
\newtheorem{prop}{Proposition}[section]
    \newtheorem{thm}{Theorem}
    \newtheorem*{thm*}{Theorem}
    \newtheorem*{cor*}{Corollary}
    \newtheorem*{prop*}{Proposition}

    \newtheorem{cor}[prop]{Corollary}
    \newtheorem{lemma}[prop]{Lemma}
\newtheorem{rmk}[prop]{Remark}
\theoremstyle{definition}
\newtheorem{defn}[prop]{Definition}
\newtheorem{notas}[prop]{Notations}
    \newtheorem{ex}[prop]{Example}

\newcommand{\K}{{\mathbb K}}
\newcommand{\N}{{\mathbb N}}
\newcommand{\Q}{{\mathbb Q}}

\newcommand{\cA}{{\mathcal A}}

\newcommand{\cE}{{\mathcal E}}
\newcommand{\cF}{{\mathcal F}}

\newcommand{\cI}{{\mathcal I}}
\newcommand{\cJ}{{\mathcal J}}
\newcommand{\cK}{{\mathcal K}}
\newcommand{\cL}{{\mathcal L}}

\newcommand{\cO}{{\mathcal O}}
\newcommand{\cP}{{\mathcal P}}

\newcommand{\cR}{{\mathcal R}}

\newcommand{\cT}{{\mathcal T}}
\newcommand{\cU}{{\mathcal U}}
\newcommand{\cV}{{\mathcal V}}
\newcommand{\cW}{{\mathcal W}}

\usepackage{bm}

\newcommand{\0}{{(0)}}

\newcommand{\e}{{\varepsilon}}

\newcommand{\nnorm}[1]{{\left\vert\kern-0.25ex\left\vert\kern-0.25ex\left\vert #1 
    \right\vert\kern-0.25ex\right\vert\kern-0.25ex\right\vert}}

\usepackage{framed,enumitem} 

\definecolor{aquamarine}{rgb}{0,0.5,0.5}

\begin{document}

\author{Marie-Claude Arnaud}
\address{Universit\'e de Paris and Sorbonne Universit\'e, CNRS, IMJ-PRG, F-75006 Paris, France. 
\& Member of the Institut universitaire de France.}
\email{marie-claude.arnaud@u-paris.fr}

\author{Jessica Elisa Massetti}
\address{Dipartimento di Matematica e Fisica, Universit\`a degli Studi Roma Tre, Rome, Italy.}
\email{ jessicaelisa.massetti@uniroma3.it}

\author{Alfonso Sorrentino}
\address{Dipartimento di Matematica, Universit\`a degli Studi di Roma ``Tor Vergata'', Rome, Italy.}
\email{sorrentino@mat.uniroma2.it}

 \title{On the fragility of periodic tori for families of symplectic twist maps}

\begin{abstract}
In this article we study the fragility of Lagrangian periodic tori for symplectic twist maps of the $2d$-dimensional annulus and prove a rigidity result for
completely integrable ones.\\
More specifically, we consider $1$-parameter families of symplectic twist maps $(f_\eps)_{\eps\in \R}$, obtained by perturbing the generating function of an analytic map $f$ by a family of potentials $\{\eps G\}_{\eps\in \R}$.
{Firstly, for an analytic $G$  and for $(m,n)\in \Z^d\times \N^*$ with $m$ and $n$ coprime,
we investigate the topological structure of the set of $\eps\in \R$ for which $f_\eps$  admits a Lagrangian periodic torus of rotation vector $(m,n)$. In particular we prove that, under a suitable non-degeneracy condition on $f$, this set consists of at most finitely many points.
Then, we exploit this to deduce a rigidity result for integrable symplectic twist maps, in the case of deformations produced by a $C^2$ potential.\\
Our analysis, which holds in any dimension, is based on a thorough investigation of the geometric and dynamical properties of Lagrangian periodic tori, which we believe is of its own interest.}
\end{abstract}

\maketitle

\section{Introduction and main results}

In the study of Hamiltonian systems, an important role is played by so-called {\it integrable systems}. These systems -- whose dynamics is quite simple to describe due to the presence of a large number of conserved quantities/symmetries --  arise quite naturally 
 {in} many physical and geometric problems. 
Investigating which of the properties of these systems  break or are preserved in the passage from the   {integrable} regime to  {non-integrable} one is a very natural and enthralling question,  which 
 {is part of what Henri Poincar\'e recognised} as  ``{\it the general problem of dynamics}'' {(see  \cite[Sec. 13]{Poincare:methodes1})}. \\
  Among the many results in this direction, 
a place of honour goes undoubtedly to Kolmogorov's breakthrough in 1954 \cite{Kolmogorov} that, together with the subsequent works by Arnol'd \cite{ArnoldKAM, ArnoldKAM2} and Moser \cite{MoserKAM},  paved the ground to what  nowadays is known as {\it KAM theory} 
(see {\cite{KAMstory} for a detailed historical account and, for example, 
\cite{BS20, EFK:2015} and \cite{CC:2009, Massetti:APDE, Massetti:ETDS}, for more recent advances and generalizations in, respectively, Hamitonian and non necessarily conservative dynamics}). \\

As a common feature, integrable systems reveal a certain \co{rigidity}: integrability appears to be a very fragile property that is not expected to persist under generic, yet small, perturbations.  
Understanding the essence of this feature is a very compelling task, which 
arises in various contexts, providing the ground for some of the foremost questions and conjectures in dynamics. 
{Let us mention that there are different possible notions of integrability {for a finite dimensional dynamical system}: one that allows singularities of the invariant foliation, that we will {refer to as {\it integrability}}, and the other that requires a regular foliation everywhere, that we call {\it complete integrability}.  \\

{
Among the most iconic questions related to integrability or complete integrability in finite dimension,  we recall the following:
\begin{itemize}[leftmargin=*]
  \item[(i)] The {\it Birkhoff conjecture} in billiard dynamics, which claims that the  only $2$-dimensional billiard tables for which the corresponding dynamics is  integrable, are elliptic ones.
  This question in the case of complete integrability has been solved in \cite{Bialy1993} (see also \cite{Wojt} by means of an integral-geometric approach).\\
  Recently there have been several breakthroughs related to the integrable case: in the perturbative setting (namely, for domains that are small perturbations of ellipses) several versions of the conjecture have been proven in \cite{ADK, HKS2018, KS2018, Koval}; in non-perturbative case, we recall the work \cite{Innami} (see also \cite{ArnBialy}) and the recent proof of the conjecture for centrally-symmetric domains in \cite{BialyMironov2020}.\\
A version of this conjecture related to the   existence of an integral of motion polynomial in the velocity ({\it algebraic Birkhoff conjecture}) was solved in \cite{Glutsyuk} (see also previous results \cite{Bolotin, BialyMironov2017}).\\
  We refer to \cite{KS2021} for a more detailed survey of these results and more references.\\
 \item[(ii)] The  problem of characterizing {\it integrable Riemannian geodesic flows} on the $d$-dimensional torus $\T^d$. For completely integrable metrics, this question is related to the so-called {\it Hopf conjecture} ({\it i.e.}, the metric must be flat) and it was solved in \cite{Hopf, BI1994}. There exist on $\T^d$ metrics that are integrable, but not completely integrable, namely metrics of the form:
 $$ ds^2 = (f_1(x_1) + f_2(x_2) + \ldots + f_d(x_d)) (dx_1^2 + dx_2^2+\ldots + dx_d^2),$$
 the so-called {Liouville metrics}. A folklore conjecture states that these metrics are the only  integrable metrics on $\T^d$.
A partial answer to this conjecture in dimension $d=2$ is provided in \cite{BFM} under the assumption that the system admits an integral of motion which is quadratic in the momenta. Observe that while the case of quadratic integral of motion reduces to a system of linear partial differential equations, the case of higher degree integrals of motions is very challenging and it turns out to be equivalent to delicate questions on non-linear partial differential equations of hydrodynamic type (see \cite{BialyMironov2011a, BialyMironov2011b, BialyMironov}).\\
A deformational version of this conjecture on $\T^2$ has been recently investigated in \cite{Henheik}.
Finally, see for instance \cite{Kolokoltsov, Bialy2010}, where this question on surfaces different from the torus is addressed. 
\end{itemize}}

{
\begin{rmk}
It is worthwhile to mention that also the study of {\it infinite dimensional integrable systems} is a very active field of research, with many aspects related to the structure of their phase space, their dynamical properties, as well as the notion of integrability itself (clearly, asking for ``infinitely many'' integrals of motion only is not a well-posed condition), being very challenging and still missing a complete understanding. 
These systems naturally arise to model a wide variety of wave phenomena ({\it e.g,}, the famous {\it KdV equation} or  the {\it Sine-Gordon equation}) or in other areas of the applied sciences, for example,  the  {\it Toda lattice},  a simple integrable model for one-dimensional crystal in solid state physics.\\ 
A detailed  review of infinite dimentional integrable systems would go well-beyond the scopes of this article. We refer interested readers to the nice expositions in \cite{Kappeler-Poschel:KAM, Kuksinbook, ZakShab74} for integrability of PDEs,  and to the ones in \cite{Toda, Kappeler-Henrici:aa} for the Toda lattice. 
\end{rmk}}

In this article we would like to shed more light on the nature of this {\it fragility} and {\it rigidity}, in the setting of symplectic twist maps of the $2d$-dimensional annulus $\T^d\times\R^d$, where $\T^d := \R^d/\Z^d, $  and $d\ge 1$.\\
More specifically, we will focus on two related aspects:
\begin{itemize}
\item[a)] The persistence and the properties of invariant tori that are foliated by periodic points  (see Definition \ref{def periodic tori}). These objects are at the core of the fragility of integrable systems, since -- as already pointed out by Poincar\'e -- they seem to be extremely easy to break and their persistence does not appear to be generic, in counterposition to the robustness of the non-periodic invariant tori considered by KAM theory. \\
More precisely, we consider a one-parameter {perturbation} of a twist map (not necessarily integrable) and investigate the topological structure of the set of parameters for which the perturbed map admits a periodic torus of a given rotation vector.
See section \ref{secintroteo1} and Theorem \ref{teo 1} for more details and precise statements.\\

\item[b)]  The rigidity of integrable twist maps, namely, to which extent it is possible to deform in a non-trivial way an integrable twist map, preserving some (or all) of its features.\\
More precisely, we consider a one-parameter perturbation of an integrable twist map and {point out dynamical conditions implying that the perturbation must be trivial.}
See section \ref{secintroteo2} and Theorem \ref{teo 2} for more details.\\
\end{itemize}

It is worth mentioning that our investigation relies on a thorough analysis of the geometric and dynamical properties of periodic tori, which -- we believe -- are interesting {\it per se} (see Sections \ref{Aactionpertori}, \ref{perlagtori}, and Appendix \ref{ALipschitzetGreen}).\\

Before stating our main results (sections \ref{secintroteo1} and \ref{secintroteo2}), in the next section let us first clarify the setting that we consider and introduce the main objects that are involved.\\

\smallskip

\subsection{Setting and Definitions} \label{setting}
    In this paper we will consider \co{symplectic twist maps} of the {$2d$-dimensional annulus $\T^d\times \R^d$, where $d\ge 1$  and $\T^d := \R^d/\Z^d$} according to the following definition.

    \begin{defn}[Symplectic twist maps]\label{defsympltwistmap}
     A { {\it symplectic twist map} } of the $2d$-dimensional annulus is a  $C^1$ diffeomorphism $f:\T^d\times \R^d\righttoleftarrow$ that admits a lift $F: {\R^{d}\times \R^d} \righttoleftarrow$, $F(q, p)=:(Q(q, p), P(q,p))$  satisfying
\begin{enumerate}
    \item $F(q +m, p)=F(q, p)+(m, 0)$ $\forall m\in\Z^d$;
    \item ({\sl Twist condition}) the map $(q, p)\mapsto(q, Q(q, p))$ is a diffeomorphism of ${\R^{d}\times \R^d}$; 
    \item ({\sl Exactness}) there exists  a {\it generating function} of the map $F$, namely a function $S:{\R^{d}\times \R^d}\rightarrow \R$ such that
    \begin{itemize}
        \item[$\bullet$] $S(q+m, Q+m)=S(q, Q)$, $\forall m\in\Z^d$,
        \item[$\bullet$] $PdQ-pdq=dS(q, Q)$.        
    \end{itemize}
\end{enumerate}
\end{defn}
  
  Note that in some of the literature, maps satisfying Definition \ref{defsympltwistmap} are often called  \textit{exact symplectic twist maps}.
  
    \begin{ex}\label{exCI} Assume that $\ell_0:\R^d\rightarrow \R$ is at least $C^2$ and that $\nabla\ell_0:\R^d\to\R^d$ is  a $C^1$ diffeomorphism.  A  \co{completely integrable} symplectic twist map is defined via
    \begin{equation}\label{ECI}
f_0 (q, p):=(q+\nabla \ell_0(p), p).
\end{equation} 
Such a map is a symplectic twist map. Moreover, $\T^d\times \R^d$ is foliated by invariant Lagrangian tori $\T^d\times \{ r_0\}$ and the restriction of $f_0$ to each of these tori is a rotation, which is periodic when  $\nabla\ell_0(r_0)\in\Q^d$.\\
The generating function of \eqref{ECI} is given by $S_0(q,Q) := h_0(Q-q)$,  where $\nabla h_0=(\nabla \ell_0)^{-1}$.
    \end{ex}
    
 In our results we will study a special family of symplectic twist maps, that we now describe.
    \begin{notas}\label{notadeformation}
Let $G:\T^d\rightarrow \R$ be at least $C^2$ and $f$ a symplectic twist map. A {\it symplectic deformation} of  $f$ by a potential $G$ is given by the family of maps
\begin{equation}\label{dpcm18ott}
f_\eps (q, p) :=   f (q, p+\eps \nabla G(q))\quad \eps\in \R.
\end{equation}
Their generating functions are given by
\begin{equation}\label{gen funct}
(q, Q)\mapsto S_\eps(q,Q):=S(q, Q)+ \varepsilon G(q),\end{equation}
where $S(q, Q)$ denotes the generating function of $f$.\\
We remark that in the context of Aubry-Mather theory for Lagrangian systems, this kind of perturbations by a potential are often called 
{\it perturbation in the sense of Ma\~n\'e} (see, for instance \cite{Mane}).
   \end{notas}

\begin{defn}[Properties of symplectic twist maps]\label{defpropsympltwistmap} 
    Let $f:\T^d\times \R^d\righttoleftarrow$ be a symplectic twist map  that admits a lift $F: {\R^{d}\times \R^d} \righttoleftarrow$, $F(q, p)=:(Q(q, p), P(q,p))$. We denote a generating function of $F$ by $S$.
  
     {\bf (i)}  The  symplectic twist map $f$ is said to be {\it positive} if there exists $\alpha>0$ such that  
     $${\partial_q\partial_Q}{S}(q, Q)(v, v)\leq -\alpha\| v\|^2 \quad \forall\, q, Q, v \in \R^d,$$    
      where $ \|.\|$ denotes the  Euclidean norm in $\R^d$.\\
      It is said to be {\it strongly positive} if there exists $\alpha, \beta>0$ 
      such that  
     $$ -\beta \|v\|^2 \leq {\partial_q\partial_Q}{S}(q, Q)(v, v)\leq -\alpha\| v\|^2 \quad \forall\, q, Q, v \in \R^d.$$
     
{\bf (ii)} It has  {\it bounded rate} if 
$\|\partial_q\partial_qS\|_\infty+\|\partial_Q\partial_QS\|_\infty$ is bounded.\\
\end{defn}

\begin{rmk} {\rm (i)}\label{rmk properties of maps}
These notions are independent of the chosen lift $F$ of $f$ and also of the chosen generating function of $F$.\\
{\rm (ii)} If we write the Jacobian matrix of $f$ at $(q, p)$ by $d$-dimensional blocks 
$$
{Df(q, p)=:}\begin{pmatrix} a(q, p)&b(q, b)\\
c(q, p)&d(q,p)
\end{pmatrix}
$$
then the positivity and strongly positivity of $f$ {can be written, respectively,}
$${\exists\,\alpha>0:} \qquad  b(q, p)(v,v)\geq {\alpha \|b(q, p) v\|^2} \quad \forall\, (q, p)\in\T^d\times \R^d,\,  \forall\, {v\in \R^d}, $$
and
$${\exists\,\alpha,\beta>0:} 
\quad 
 {\beta}\| b(q,p) v\|^2\geq b(q, p)(v,v)\geq {\alpha}\| b(q, p) v\|^2
\quad  \forall\, (q, p)\in\T^d\times \R^d,  \forall\,{v\in \R^d}.  
 $$
{Moreover, the bounded rate condition corresponds to the boundedness of $\| b^{-1}a\|_\infty+\| db^{-1}\|_\infty$, which can be interpreted by saying that}  the direct and inverse images of the verticals $\{ 0\}\times \R^d$ by {$Df(q,p)$} are graphs of linear maps $L^\pm_{(q, p)}:\R^d\to\R^d$ whose norms $\| L^\pm_{(q, p)}\|$ are uniformly bounded.\\
{\rm (iii)} 
{If we consider a symplectic deformation of a map $f$ by a potential, as introduced in \eqref{dpcm18ott}, we observe that
if $f_0$ is positive (respectively, strongly positive/with bounded rate), then all the diffeomorphisms of the family  $(f_\eps)_{\eps\in \R}$  are also positive (respectively, strongly positive/with bounded rate).
}
\end{rmk}

\smallskip

\begin{rmk}\label{rmkciboundedrate} Coming back to Example \ref{exCI}, $f_0$ is positive if $D^2\ell_0 (p)$ is positive definite and there exists $\alpha>0$ such that $$D^2\ell_0 (p) (v,v) {\geq}  \frac{1}{\alpha }\|v\|^2  \quad \forall\, p, v \in \R^d.$$
Similarly, $f_0$ is strongly positive if there exist $\alpha, \beta>0$ such that 
$$ \frac{1}{\alpha} \|v\|^2  {\leq}  D^2\ell_0 (p) (v,v)  {\leq}   \frac{1}{\beta} \|v\|^2  \quad \forall\, p, v \in \R^d.$$
Observe that a positive completely integrable symplectic twist map has always bounded rate.
\end{rmk}

\smallskip

{Our  study will focus on the existence and the properties of the following  dynamical objects: {\it periodic} and {\it completely periodic graphs} of a symplectic twist map}.

\begin{defn}[Periodic and completely-periodic tori]\label{def periodic tori}
Let $F:\R^d\times \R^d\righttoleftarrow$ be a lift of a symplectic twist map $f:\T^d\times \R^d\righttoleftarrow$. Let {$\gamma: \R^d \longrightarrow \R^d$} be a $\Z^d$-periodic and {continuous} function, and let 
  $\cL:={\rm graph}(\gamma)$. For $(m, n)\in \Z^d\times\N^*$ with $m$ and $n$ coprime, we say that:
  \begin{itemize}
  \item[{\bf (i)}]
   $\cL$ is a {\it $(m, n)$-periodic graph} of $F$ if
$$ F^n(q, \gamma (q))=(q+m, \gamma (q)) \qquad \forall \,q\in\R^d;$$
  \item[{\bf (ii)}]
$\cL$ is  a {\it $(m, n)$-completely periodic graph} of $F$ if it is invariant by $F$ and a $(m, n)$-periodic graph of $F$.
\end{itemize}
{We might refer to the projection of $\cL$ to $\T^d\times \R^d$ as a {\it periodic} (resp., {\it completely periodic}) {\it torus} of $f$.}
{Then a  $(m,n)$-periodic graph $\cL$ satisfies $f^n(\cL)=\cL$ and a $(m, n)$-completely periodic torus is $f$-invariant.}
\end{defn}

\begin{rmk} 
Coming back to Example \ref{exCI},  if $F_0 (q, p):=(q+\nabla \ell_0(p), p)$ is a lift of $f_0$, 
then, for every $p_0\in \R^d$ such that 
$\nabla \ell_0(p_0) = \frac{m}{n}\in \Q^d$, where $(m, n)\in \Z^d\times\N^*$ with $m$ and $n$ coprime, the graph $\cL_{p_0}:=\R^d \times \{p_0\}$ is a $(m,n)$-completely periodic graph of $F_0$.\
\end{rmk}

{In the following, the periodic invariant graphs we will look at are {\it Lipschitz} (or $C^0$) {\it  Lagrangian graphs}. More precisely:}

\begin{defn}
{Let $\gamma: \R^d \longrightarrow \R^d$ be a {Lipschitz} function.  Then
$\cL:={\rm graph}(\gamma)$ is said to be a {\it Lagrangian} graph if for every 
$C^1$ loop $\nu: \T \longrightarrow \R^d$, we have}
$$\int_\T\gamma(\nu(t))\dot\nu(t)dt=0.$$
\end{defn}

\medskip

\begin{rmk} \label{G20}
{{\bf (i)} Let ${\rm graph}(\gamma)$ be a Lipschitz Lagrangian graph. If we denote $\gamma(q):=(\gamma_1(q), \ldots, \gamma_d(q))$,
 it follows from the definition of being Lagrangian that the $1$-form $\sum_{i=1}^d \gamma_i(q)\,dq_i$ is exact in $\R^d$.
In particular,  if $\gamma$ is $\Z^d$-periodic, then $\gamma = c+du$ where $c\in\R^d$ and $u:\R^d\rightarrow\R$ is  $C^{1,1}$ and $\Z^d$-periodic.\\
{\bf (ii)} Observe that the limit of a family of uniformly Lipschitz Lagrangian graphs is also a Lipschitz Lagrangian graph.\\
{\bf (iii)}  We will prove in Proposition \ref{invarianttori} that for positive symplectic twist maps, if one considers Lipschitz Lagrangian graphs, then the notions of periodic and completely periodic graphs coincide.}
\end{rmk}

\smallskip

{To conclude this subsection, let us introduce a regularity assumption that will be assumed in our main results (Theorems \ref{teo 1} and \ref{teo 2}), namely that 
the unperturbed symplectic twist map $f_0$ to satisfy the following {\it analyticity condition}.}

\begin{defn}[{Analyticity property}] \label{Danaly} 
A  symplectic twist map $f: \T^d\times \R^d\righttoleftarrow $ satisfies the {\it analyticity property} if there exists a holomorphic map $\cF: \C^d\times \C^d \righttoleftarrow$, where $\cF(q, p)=:(Q(q, p), P(q,p))$, such that: 
\begin{enumerate}
    \item $\cF$ is a holomorphic  diffeomorphism of $ \C^d\times \C^d $;
    \item $\cF_{|\R^{d}\times \R^d}$ is a lift of $f$;
    \item ({\sl Twist condition}) the map $(q, p)\mapsto(q, Q(q, p))$ is a diffeomorphism of $\C^d\times \C^d$; 
    \item ({\sl Exactness}) there exists a { generating function} $S: \C^d\times \C^d  \rightarrow \C$ such that
    \begin{itemize}
        \item[$\bullet$]  $S(q+m, Q+m)=S(q, Q)$ $\forall m\in\Z^d$;
        \item[$\bullet$] $PdQ-pdq=dS(q, Q)$.        
    \end{itemize}
\end{enumerate}
\end{defn}

\begin{rmk}
Coming back to Example \ref{exCI},
let $\ell_0: \C^d  \rightarrow \C$ be an analytic function such that
${\ell_0}_{|\R^d}$ is real and  $\nabla \ell_0:  \C^d \rightarrow \C^d$ is a diffeomorphism. Then, $f _0(\theta, r)= (\theta+\nabla\ell_0(r), r)$ is a completely integrable symplectic twist map of $\T^d\times \R^d$  that satisfies the analyticity property as in Definition \ref{Danaly}.
\end{rmk}

\bigskip

\subsection{Main results}

{In this section we will state our two main results, which are both concerned with the existence of  $(m,n)$-periodic tori for families of  symplectic twist maps obtained as symplectic deformations by potentials (see  \eqref{dpcm18ott}) of a map $f$ which will be assumed to be strongly positive and to satisfy the analyticity property.}\\

{We can briefly outline the content of our results as follows:\\}

{The first result (Theorem \ref{teo 1}) will investigate, for  fixed  $(m,n)\in \Z^d\times \N^*$ with $m$ and $n$ coprime,  the ``{size}'' of the set of parameters $\eps\in \R$ for  for which the corresponding perturbed map in the symplectically deformed family admits a  {Lipschitz} Lagrangian $(m,n)$-periodic  graph. As far as the regularity of the  the perturbing potential is concerned, we will assume that it admits a holomorphic extension to $\C^d$.\\
}

In the second result  (Theorem \ref{teo 2}), we will be interested in the ``rigidity'' of completely integrable symplectic twist maps (see Example \ref{exCI}), proving that any symplectic deformation by a non-constant potential cannot preserve integrability or part of it (we will provide a precise description of this ``trace'' of integrability that guarantess rigidity). The regularity assumptions on the perturbing potential will be relaxed (if compared to Theorem \ref{teo 1}), not requiring any analyticity assumption.\\

\medskip

\subsubsection{\bf Main Results I: on the fragility of periodic tori} \label{secintroteo1}
Let us state our first main result concerning the possible existence and persistence of Lagrangian $(m,n)$-periodic tori -- for a given $(m, n)\in \Z^d\times\N^*$ with $m$ and $n$ coprime -- in the case of a $1$-parameter family of symplectic twist maps, obtained
as symplectic deformations by a potential, as defined in \eqref{dpcm18ott}.
We remark that our maps will satisfy suitable non-degeneracy and regularity assumptions ({\it i.e.}, analyticity, see definition \ref{Danaly}), but  no extra dynamical requirement  will be imposed.\\

\medskip

\begin{thm}\label{teo 1}
Let $f: \T^d \times \R^d \righttoleftarrow$ be symplectic twist map and let $F: \R^d \times \R^d \righttoleftarrow$ be a lift  of $f$ and let $S: \R^d\times \R^d \longrightarrow \R$ be its generating function.  Let also $G: \T^d\to\R$ be a potential function.  \\
Consider the family of symplectic twist maps $f_\eps:  \T^d \times \R^d \righttoleftarrow$, with $\eps\in \R$, obtained as symplectic deformation of $f$ by $G$, as in \eqref{dpcm18ott},
and denote by $F_\eps$ a continuous family {of lifts of $f_\eps$}.\\
{\rm I.} Assume that:
\begin{enumerate}[(i)]
\item  {$f$ is strongly positive},
\item  $f$  satisfies the analyticity property,
\item  $G$  admits a holomorphic extension to $\C^d$.
\end{enumerate}
Then,  for every $(m,n)\in \Z^d\times \N^*$, with $m$ and $n$ coprime, the set
\[
 \set{\eps\in \R: \, F_\eps \text{ has a {Lipschitz} Lagrangian $(m,n)$-periodic  graph}}
\]
is either the whole $\R$ or consists of isolated points.\\
{\rm II.} If, in addition, $f$ has also bounded rate and $G$ is non-constant, then the above set consists of at most finitely many points.\\
\end{thm}

\begin{rmk}
It is not necessary in the proof that $F$ admits a holomorphic extension to the whole $\C^d\times \C^d$ (see Definition \ref{Danaly}):
it would be enough that the analyticity property holds
on a strip $\Sigma^{2d}_\sigma$, where 
$\Sigma_\sigma:=\{z\in \C: \; |{\rm Im} \,z| <\sigma\}$ for some $\sigma>0$. Similarly, $G$ can be assumed to admit a holomorphic extension only to a strip and the dependence on $\varepsilon$ being at higher orders, {\it i.e.} $G(q,\varepsilon)$.
However, we decided not to pursue this further generality, in order to ease the notation and make the presentation clearer.\\
\end{rmk}

One can deduce from Theorem \ref{teo 1} the following rigidity result.
\begin{cor}\label{cor 1}
Assume the hypotheses and the notations of Theorem \ref{teo 1}, including $f$ being bounded-rate.
{Let $(m,n)\in \Z^d\times \N^*$, with $m$ and $n$ coprime, such that the maps $f_\e$ admit a Lipschitz Lagrangian $(m,n)$-periodic torus for infinitely many values of $\eps \in \R$}. Then, $G$ must be constant.
\end{cor}

To the best of our knowledge, Theorem \ref{teo 1} (as well as Corollary \ref{cor 1}) is the first result of this kind for symplectic twist maps. 
We underline that the analysis of the persistence of Lagrangian periodic tori for a {\it single} rotation vector $(m,n)$, which seems to be a weak requirement, is made possible by the regularity assumption on the family of symplectic twist maps, namely analyticity.\\
In the next section, we will discuss how to weaken this latter assumption and still obtain a similar rigidity result.\\
For 2-dimensional  Birkhoff billiards, some results concerning the preservation and the existence of periodic caustics have been discussed in \cite{KZ, KK2021, Zhang}.

\begin{rmk}
(This remark has been pointed out by one of the referees).
Theorem \ref{teo 1} and Corollary \ref{cor 1} could be read in the usual framework of rigidity theories and moduli spaces. In the space of all symplectic twist maps with bounded rate satisfying the analyticity condition, one can define a natural ``{\it broken line}'' equivalence relation. Given $f_0, f_1$ symplectic twist maps with bounded rate and satisfying the analyticity condition, we 
say that $f_0 \sim f_1$ if:
\begin{itemize}[leftmargin=*]
\item[$\bullet$] there exists $G: \T^d\to \R$, analytic with holomorphic extension to $\C^d$, such that the generating function of $f_1$ is given by $S+G$, where $S$ is the generating function of $f_0$,
\item[$\bullet$] there is a coprime pair $(m, n)\in \Z^d\times \N^*$, such that for an infinite set of $\varepsilon$'s, the map $f_\varepsilon$ generated by $S+ \varepsilon G$ admits an $(m,n)$- Lagrangian periodic torus.
\end{itemize}
We say that  $f$ is equivalent to $g$ if there exists a sequence of maps $f_0=f, f_1, \ldots, f_n=g$ such that $f_k \sim f_{k+1}$. This is an equivalence relation and
 Corollary \ref{cor 1} implies that it is trivial, {\it i.e.}, equivalence classes reduce to singletons.
\end{rmk}

\subsubsection{\bf Main results II: on the rigidity of integrable symplectic twist maps}\label{secintroteo2}

Our original motivation for this work was to investigate the rigidity of  integrable symplectic twist maps, namely whether to understand under which deformations of the map it is possible to preserve integrability or some of its features. As we mentioned before, this problem is in fact related to important conjectures in billiard and Hamiltonian dynamics, and such an investigation seems to be a preliminary step in order to tackle them.\\

Let us state our second main result.

\begin{thm}\label{teo 2} 
Let $f: \T^d \times \R^d \righttoleftarrow$ be symplectic twist map and
let $F: \R^d \times \R^d \righttoleftarrow$ be a lift  of $f$ and $S: \R^d\times \R^d \longrightarrow \R$ its generating function.\\
Let $G\in C^2(\T^d)$. 
Consider the family of symplectic twist maps $f_\eps:  \T^d \times \R^d \righttoleftarrow$, with $\eps\in \R$, obtained as symplectic deformation of $f$ by $G$, as in \eqref{dpcm18ott},
and denote by $F_\eps$ a continuous family {of lifts of $f_\eps$}.\\
Assume that:
\begin{enumerate}[(i)]
\item   $f$ is completely integrable (as in Example \ref{exCI}); 
\item  {$f$ is strongly positive}; 
\item $f$  satisfies the analyticity property; 
\item  \label{weak integr}
there exist a basis  $(q_1,\ldots, q_d)$ of $\Q^d$ and $I_1,\ldots, I_d \subset \R$  open intervals, such that for any  $\frac{m}{n}\in \bigcup_{j=1}^d I_jq_j \cap \Q^d$,
$F_\eps$ has a {Lipschitz} Lagrangian $(m,n)$-periodic  graph  {for infinitely many values of $\eps\in \R$, accumulating to $0$}.
\end{enumerate}
Then, $G$ must be identically constant.
\end{thm}

\medskip

\begin{rmk}
Differently from Theorem \ref{teo 1}, we do not  ask any extra assumption on the regularity of $G$. In fact, 
as we will see in Proposition \ref{piazzetta},
the assumption 
 on the existence of ``sufficiently many'' Lagrangian periodic graphs for {$(F_\eps)_{\eps \in \R}$}
 {\rm (}in the sense of item \ref{weak integr}{\rm )} implies that $G$ must be a trigonometric polynomial, thus allowing us to apply Theorem \ref{teo 1}.
Note that being $f$  completely integrable,  the hypothesis of strong positivity automatically implies that $f$ is of bounded-rate (see also Remark \ref{rmkciboundedrate}).
{We remark that it is important in the proof that the sets of $\eps\in \R$ (which might vary for different rotation vectors) accumulate to $0$ (see Lemma \ref{lemmino base})}.\\
\end{rmk}

{For a given $\eps \in \R$, the assumption on $F_\eps$ in item \ref{weak integr} is satisfied, for example, if there exists
an open set ${\mathcal A} \subset \R^d$ 
such that $F_\eps$ has a {Lipschitz} Lagrangian $(m,n)$- periodic  graph  for any  $\frac{m}{n}\in {\mathcal A} \cap \Q^d$. }
This property can be considered as a weaker notion of integrability, in some literature called
{\it weak rational integrability} (see for instance \cite{BialyMironov2020, KS2018, HKS2018}).\\
In order to see that the condition of item \ref{weak integr} is satisfied, one can choose $q_1,\ldots, q_d \in \Q^d$ linearly independent, such that the half-lines
\begin{eqnarray*}
\sigma_{q_i}: [0,\infty) &\longrightarrow& \R^d\\
t &\longmapsto& tq_i
\end{eqnarray*}
intersect ${\mathcal A}$ (this is possible since $\mathcal A$ is open).
For every $j=1,\ldots, d$, choose $0<a_j<b_j$ such that $\sigma_{q_j}((a_j,b_j)) \subset \mathcal A$, and let $I_j:=(a_j,b_j)$.\\

\begin{rmk} 
Note that in dimension $d=1$  weak rational integrability implies  the existence of an open set foliated by invariant curves  ({\it i.e.}, local $C^0$-integrability).
 Although this property is not needed for our results, we provide a sketch of its proof for interested readers (see also  \cite[Theorem 3]{MS}, where a similar argument has been used).\\
By assumption,  for  every $\frac p q \in (a,b) \cap \Q$, with $a<b$, there exists an invariant curve of rotation number $\frac p q$ that is foliated by periodic points;  this curve is the Lipschitz graph of a function  $\gamma_{p/q}: \T \longrightarrow \R$.
Applying Arzel\`a-Ascoli theorem (these graphs are equiLipschitz, see Proposition \ref{lemmac1} (ii)), one can obtain an invariant curve for every rotation number $h\in (a,b)$, which is the graph of a Lipschitz function  $\gamma_{h}: \T \longrightarrow \R$.
Observe that all these graphs are disjoint and that for every $h\in (a,b)$ there exists a unique invariant curve of rotation number $h$ (when $h$ is rational, uniqueness follows from the fact that all of its points are periodic).\\
We denote by $\Lambda$ the union of the graphs of $\gamma_h$ for $h\in (a,b)$.
 We want to show that $\Lambda$ is open. Let us define 
\begin{eqnarray*}
 \Phi : \T \times (a,b) &\longrightarrow& \T\times \R\\
  (q, h) &\longmapsto& (q,\gamma_h(q)).
 \end{eqnarray*}
One can prove that  $\Phi: \T\times (a,b) \to \T\times \R$ is continuous and  injective, hence it is an open map
(it follows from the invariance of domain theorem, see \cite[Theorem 2B.3]{Hatcher}). Therefore, $\Lambda=\Phi(\T\times (a,b))$ is open.\\
\end{rmk}

Hence, we can deduce the following Corollary from Theorem \ref{teo 2}.

\begin{cor} Under the hypotheses $(i)-(iii)$ and the notations of Theorem \ref{teo 2},  and assuming (instead of (iv))
\begin{itemize}
\item[(iv$^*$)] there exists an open set ${\mathcal A} \subset \R^d$ 
{such that, for infinitely many $\eps\in \R$ accumulating to $0$, the map $F_\eps$ has a {Lipschitz} Lagrangian $(m,n)$- periodic  graph for any  $\frac{m}{n}\in {\mathcal A} \cap \Q^d$},
\end{itemize}
then,  $G$ must be identically constant.

\end{cor}

\bigskip
\begin{rmk}\label{BialyMcKaySuris}
{\rm (i)} {Bialy and MacKay  proved in \cite{BialyMcKay2004} that a generalized standard map of $\T^d\times \R^d$ that has no conjugate point\footnote{
We recall that, for a map $f:\T^d\times \R^d \righttoleftarrow$, two points $(q, p)$ and $f^n(q, p)$ with $n\not=0$ are said to be {\it conjugate} if 
$ Df^n(q, p) \big(\{0\}\times \R^d\big)\cap\big(\{ 0\}\times \R^d\big) \neq \{ (0,0)\}.$}
 (e.g., a {completely} integrable symplectic twist map has no conjugate point) has to be of the form $(q, p)\mapsto(q+p, p)$. 
Their proof uses in a crucial way the fact that  the whole $2d$-dimensional annulus is foliated by invariant tori and then it does not imply our results, even for their choice of a particular function $\ell_0$.\\
{{\rm (ii)}}
In dimension 2, Suris \cite{Suris} exhibited an example of a generalized standard map that is {integrable} {(for all values of the parameter $\epsilon$ for which it is defined)}, but his example has singularities and does not satisfy the assumptions of Theorem \ref{teo 2}.}\\
{{\rm (iii)} In a different direction,  Chen, Damjanovi\'{c} and Petrovi\'c   investigate in \cite{CDP} sufficient conditions for an exact symplectic twist map to be integrable; in particular, they show that if such a map
is close to $(q, p)\mapsto (q+p, p)$ and  commutes with a map that is close to $(q, p)\mapsto (q+\alpha, p)$ and  semi-conjugate to the Diophantine rotation $q\mapsto q+\alpha$, then it is integrable.\\
}
\end{rmk}

\medskip

\subsection*{Acknowledgements}
J.E.M and A.S. acknowledge the support of the INdAM-GNAMPA grant ``{\it Spectral and dynamical properties of Hamiltonian systems}''. J.E.M. has been also supported by the
research project PRIN 2020XBFL ``{\it Hamiltonian and dispersive PDEs}'' of the Italian Ministry of Education
and Research (MIUR).
A.S. has been also supported by the
research project  PRIN Project 2017S35EHN ``{\it Regular and stochastic behavior in
dynamical systems}''  and  by  the MIUR Department of Excellence grant 2023-27 MatMod@TOV.
J.E.M. and A.S. would like to thank Stefano Marmi and  Laurent Stolovitch for  enriching comments and discussions. A.S. also wishes to thank Stefano Trapani for very helpful discussions. The three authors are grateful to Jianxing Du, Xifeng Su and Philippe Thieullen for pointing out a mistake  in Proposition B.1 in a previous version of the article. Finally, they are grateful to the anonymous referees for their careful reading of the paper and for very helpful comments. \\

\medskip


\section{ {Some properties of}  $(m, n)$-periodic graphs}\label{Aactionpertori}

In this section we collect and prove several properties of $(m, n)$-periodic tori, which will be used in the proofs of our main results, as well as being  interesting {\it per se}.\\
We shall discuss:
\begin{itemize}
\item[$\bullet$] Action-minimizing properties of Lipschitz Lagrangian $(m, n)$-periodic graphs (section \ref{subsecactionminprop}). In particular, we prove that:
\begin{itemize} 
\item orbits starting on a Lipschitz  Lagrangian $(m, n)$-periodic graph have all the same action (Proposition \ref{torusactionminimizing})  and, under suitable growth conditions on the generating function, they are action-minimizing (Proposition \ref{p(m,n)minimizing});
\item for positive symplectic twist maps, there is at most on Lipschitz Lagrangian $(m,n)$-periodic graph (Proposition \ref{uniqtorus});
\item  for positive symplectic twist maps, Lipschitz Lagrangian $(m,n)$-periodic graphs are indeed invariant, hence completely periodic (Proposition \ref{invarianttori}).\\
\end{itemize}

\item[$\bullet$] Regularity properties of Lipschitz Lagrangian $(m, n)$-periodic graphs (section \ref{secregularityLaggraphs}). In particular, we prove that:
\begin{itemize}
\item  for positive symplectic twist maps, Lipschitz Lagrangian $(m,n)$-periodic graphs are  as regular as the maps is (Proposition \ref{Plocaltwist}).\\
\end{itemize}

\item[$\bullet$] Symplectic properties of $(m, n)$-periodic graphs (section \ref{perlagtori}). In particular:
\begin{itemize}
\item  we provide sufficient conditions that ensures that a $(m, n)$-periodic graph is  Lagrangian (Proposition \ref{P(m,n)Lagr}).\\
\end{itemize}

\end{itemize}

\subsection{Action-minimizing properties of Lipschitz Lagrangian $(m, n)$-periodic graphs} \label{subsecactionminprop}

In the following, let us assume that $F:   \R^{d}\times \R^d \righttoleftarrow$ 
is a lift of a symplectic twist map  $f: \T^d\times \R^d \righttoleftarrow$   with generating function $S(q,Q)$.\\

Let us start with the definition of {\it action}.

\begin{defn}[Action]
Let $n\in \N$. Given a finite sequence $(q_0, \ldots, q_n)\in \R^{d(n+1)}$, we define its {\it action} as:
 $$
 {\cA(q_0,\ldots, q_n)} := \sum_{j=0}^{n-1} S(q_j,q_{j+1}).
 $$
\end{defn}

\begin{rmk}\label{variationalprinciple}
It is a classical result in the study of symplectic twist maps that a sequence $(q_0, \ldots, q_n) \in \R^{d(n+1)} $ is the projection of an orbit segment of $F$ on the $q$-component
(i.e., there exists $(q_0,p_0) \in \R^d\times \R^d$ such that $q_j = \pi_1 \circ F^j(q_0,p_0)$, for every $j=1,\ldots, n$, where $\pi_1(q, p):=q$),
if and only if it is a critical point of $\cA$ restricted to the subspace of sequences $(w_0, \ldots, w_n)\in(\R^d)^{n+1}$ with fixed endpoints $w_0=q_0$ and $w_n=q_n$. See for instance \cite[Corollary 5.5]{Gole}.\\
\end{rmk}

\begin{defn}[Action-minimizing sequences and orbits]\label{Dminimizer}
A sequence $(q_j)_{j\in\Z}\in(\R^d)^\Z$ is  called {\it action-minimizing} (or simply {\it minimizing}) if for any $k,h \in \Z$, with $k<h$,   the sequence $(q_k, \ldots, q_h)$ minimizes the action on the subspace of sequences $(w_0, \ldots, w_n)\in(\R^d)^{n+1}$ with fixed endpoints $w_0=q_0$ and $w_n=q_n$.\\
 {An orbit of $F$ is said to be action-minimizing, if its projection on the first component is an action-minimizing sequence.}\\
\end{defn}

Let us first prove the following  result on the action of  sequences obtained as projection, on the $q$-coordinate, of orbits starting on Lipschitz Lagrangian $(m,n)$-periodic graphs.\\

\begin{prop}\label{torusactionminimizing}
We assume that $\cL$ is a  {Lipschitz} Lagrangian $(m,n)$-periodic  graph.   
 For every $(q, p)\in \cL$,   {if we denote its orbit by} $(q_j(q,p), p_j(q,p)):=F^j(q, p)$, $j=0, \ldots, n$,  then the function
$$
 {\cW(q,p)}:= \cA(q, q_1(q,p), \ldots, q_n(q,p))
$$
is constant on $\cL$.  
\end{prop}

\begin{proof}
Since $\cL$ is a {$\Z^d$-periodic Lipschitz Lagrangian graph}, it corresponds to the graph of $c+du$, where $c\in\R^d$ and $u:\R^d\rightarrow \R$ is a $\Z^d$-periodic   $C^{1,1}$  function (see Remark \ref{G20}).
Let us consider the function $\cE: (\R^d)^{n+1}\rightarrow \R$ defined by 
$$\cE(q_0, \dots, q_n):=c\cdot (q_0-q_n) +u(q_0) - u(q_n) +\sum_{j=0}^{n-1}S(q_{j}, q_{j+1}) .$$

A point $(q_0, \dots, q_n)$ is a critical point of $\cE$ if and only if 
\begin{equation}\label{critpointeq1}
c+\nabla u(q_0)=-\partial_q S(q_0, q_1),\quad  c+\nabla u(q_n)= \partial_Q S(q_{n-1}, q_n)
\end{equation}
and
\begin{equation}\label{critpointeq2}
 \partial_Q S(q_{j-1}, q_j)+\partial_qS(q_j, q_{j+1})=0 \qquad \forall j\in \{ 1, \ldots, n-1\}.
\end{equation}
The two first equations mean that $$(q_0, -\partial_q S(q_0, q_1))\in\cL\quad {\rm and}\quad (q_n, \partial_Q S(q_{n-1}, q_n))\in\cL,$$ while the latter ones state that $q_j = \pi_1 \circ F^j (q_0, -\partial_q S(q_0, q_{1}))$ for every $j=0,\ldots, n$, 
where $\pi_1(q, p):=q$ denotes the projection on the $q$-component,
hence, they correspond to an orbit of $F$.
Moreover, since 
$\cL$ is $(m, n)$-periodic, we also have that $q_n=q_0+m$ and, since $u$ is $\Z^d$-periodic,  $u(q_n)=u(q_0)$. Therefore, if $(q_0, \ldots, q_n)$ is a critical point, then 
\begin{equation}\label{Chopin}
\cE(q_0, \dots, q_n)=-c\cdot m+\cW(q_0,c+\nabla u (q_0)).
\end{equation}
In particular, if $(q, c+\nabla u(q))\in \cL$ and $q_j:=\pi_1\circ F^j(q, c+\nabla u(q))$, $j=0, \ldots, n$, the corresponding sequence of points $(q_0,\ldots, q_n)$ give rise to a subset of $(\R^d)^{n+1}$ consisting of critical points of $\cE$. This implies 
that $\cE$ is constant on this set. Using \eqref{Chopin}, one concludes that  $\cW(q,p)$ restricted to $\cL$  is constant.
\end{proof}

\bigskip

Let us now prove the following property of action-minimizing property for  Lipschitz Lagrangian $(m, n)$-periodic graphs (see also  \cite[Theorem 35.2]{Gole} and \cite[Appendix 2]{McKMS89} where it is attributed to Herman).\\
 
 \begin{prop}\label{p(m,n)minimizing}
 Let $\cL$ be a  {Lipschitz}  Lagrangian $(m, n)$-periodic  graph for $F$. Assume that $F$ admits a generating function  $S(q,Q)$ that satisfies the following superlinearity condition
\begin{equation}\label{Esuperlin}\lim_{\| Q-q\|\rightarrow +\infty} \frac{S(q, Q)}{\| Q-q\|}=+\infty.\end{equation}
 Then, every orbit of $F$ starting at a point of  $\cL$ is action-minimizing.
 \end{prop}
 
{

 \begin{proof}
 We use the same notation that as in Proposition \ref{torusactionminimizing}. For $q_0,\ldots, q_n \in \R^d$ let
 $$\cE(q_0, \dots, q_n)=c\cdot (q_0-q_n)+u(q_0) - u(q_n) +\sum_{j=1}^n S(q_{j-1}, q_j).$$
 Observe that $$\cE(q_0+ \nu, \dots, q_n+\nu)=\cE(q_0, \dots, q_n) \qquad \forall \nu \in\Z^d;$$
 moreover, the superlinearity hypothesis implies that 
 $$\lim_{\max_{1\leq j\leq n}\|q_j-q_{j-1}\|\rightarrow +\infty}\cE(q_0, \dots, q_n)=+\infty.$$
 Hence, $\cE$ has a minimum, which is attained on projections of orbits that start on $\cL$, as
 all  critical points of $\cE$ (see the proof of Proposition \ref{torusactionminimizing}).
 Observe that for any $k\in\N^*$, the hypotheses of the proposition are satisfied for $(km, kn)$ replacing $(m, n)$. We deduce that every piece of projected orbit with length $kn$ that starts on $\cL$ has minimal action amongst all sequences with the same length and the same end-points.
This implies that every piece of projected orbit of a point of $\cL$ is indeed an action-minimizer  {(otherwise one could reduce the action of a suitable segment of projected orbit of length $kn$, for $k$ sufficiently large, contradicting its minimality)}.

 \end{proof}

\begin{rmk}\label{twist}
{\bf (i)} The superlinearity condition \eqref{Esuperlin} is satisfied by any positive symplectic twist map $f: \T^d\times \R^d \righttoleftarrow$ (see, for example,  \cite[Lemma 27.2]{Gole}). \\
\noindent {\bf (ii)}    Recall that two points $(q, p)$ and $F^n(q, p)$ with $n\not=0$ are said to be {\it conjugate} if 
$$ DF^n(q, p) \big(\{0\}\times \R^d\big)\cap\big(\{ 0\}\times \R^d\big) \neq \{ (0,0)\}.$$
It is known that  along every  {action-minimizing orbit of a  positive symplectic twist map}, there are no conjugate points (see \cite[Proposition 6]{Arnaud:2013} or \cite{BialyMcKay2004}).
Therefore, if $\cL$ is a  {Lipschitz}  Lagrangian $(m, n)$-periodic  graph for $F$, then for every $(q,p)\in \cL$, if we denote by $(q_n, p_n):=F^n(q, p)$, we have $\det(\partial_pq_n)\neq0$.\\
 \end{rmk}

\bigskip

Using the action-minimizing property proved in Proposition \ref{p(m,n)minimizing}, it follows the following uniqueness result for Lipschitz Lagrangian  $(m, n)$-periodic graphs.\\

   \begin{prop}\label{uniqtorus} Assume that $f: \T^d\times \R^d \righttoleftarrow$ is a  positive symplectic twist map. \\
\noindent  {\bf {(i)}}  For every $(m,n)\in\Z^d\times \N^*$ and  for every $q\in \R^d$, there exists at most one $p\in \R^d$ such $F^n(q, p)=(q+m, p)$ and the {corresponding orbit starting at $(q, p)$ is action-minimizing}.\\
     {\bf {(ii)}}  For every $(m,n)\in\Z^d\times \N^*$, there is at most one {Lipschitz} Lagrangian  $(m, n)$-periodic graph of $F$.
 \end{prop}
 
 \medskip

 \begin{proof}
 {\bf (i)} { Let us assume that two such points $(q, p)$ and $({q}, P)$ exist. We use the notation $(q_j, p_j)=F^j(q, p)$ and $(Q_j, P_j)=F^j({q}, P)$. 
 Then, the sequence $(q,q_1, \dots, {q_n}=q+m,Q_1+m, Q_2+m, \dots , Q_n+m=q+2m) $ is such that its action equals 
 $$\cA(q,q_1,\dots,  q_n)+\cA(Q_n,Q_{n+1},  \dots , Q_{2n})
 =\cA(q_0, \dots, q_{2n}).
 $$
 We  deduce that this sequence is action-minimizing and hence the projection of an orbit (see Remark \ref{variationalprinciple}). In particular:
 $$
 \partial_Q S(q_{n-1}, q_n)+\partial_qS(q+m, Q_1+m)=0  
 $$
or equivalently
$ \partial_Q S(q_{n-1}, q_n) = - \partial_q S(q+m, Q_1+m)$. It follows from the definition of generating function that  $p=P$.\\
}

\noindent {\bf (ii)}  We know that every orbit { that starts} on a Lipschitz Lagrangian  $(m, n)$-periodic graph of $F$ is action-minimizing (see Proposition \ref{p(m,n)minimizing}). Hence the conclusion is a consequence of what has been proven in item (i).
 \end{proof}

\bigskip

Finally, let us prove that for  positive symplectic twist maps, Lipschitz Lagrangian $(m, n)$-periodic  tori are indeed invariant, hence they are $(m,n)$-completely periodic.

\begin{prop}\label{invarianttori}
Assume that $f: \T^d\times \R^d \righttoleftarrow$ is a  positive symplectic twist map. Then,
every Lipschitz Lagrangian $(m, n)$-periodic   graph of $F$  is invariant.
\end{prop}

 \begin{proof}
 We denote by $\cL$ the Lagrangian $(m, n)$-periodic  graph of $F$.  Let us consider $(q, p)\in\cL$ and let us prove that $F(q, p)\in\cL$. We know from Proposition \ref{p(m,n)minimizing} that the projected orbit of $(q,p)$ is action-minimizing. Hence, the  projected orbit starting at $F(q,p)=(q_1, p_1)$ is also action-minimizing. Moreover, there exists $p_2\in\R^d$ such that $(q_1, p_2)\in \cL$ and the projected orbit of $(q_1, p_2)$ is also action-minimizing (Proposition \ref{p(m,n)minimizing}). By Proposition \ref{uniqtorus} (i), we deduce that $p_1=p_2$ and therefore $F(q,p)=(q_1, p_1) \in \cL$, thus concluding that $\cL$ is invariant for $F$.
 \end{proof}
 
}

\bigskip

\subsection{Regularity of  Lipschitz Lagrangian $(m, n)$-periodic graphs} \label{secregularityLaggraphs}
Let $F:   \R^{d}\times \R^d \righttoleftarrow$ 
be a lift of a symplectic twist map  $f: \T^d\times \R^d \righttoleftarrow$. In this subsection we prove that Lipschitz Lagrangian $(m, n)$-periodic graphs of $F$ are as regular as $F$ is.

Let us start with the following lemma.\\

\begin{lemma}\label{lemreg}
Let $\Phi: \R^d \times \R^d \righttoleftarrow$ be the lift of a $C^k$ diffeomorphism $\varphi: \T^d \times \R^d \righttoleftarrow$, $k\geq 1$, $k=\infty$, or analytic. 
Let $\eta:\R^d \to \R^d$ a continuous $\Z^d$-periodic function and  denote $\cL := {\rm graph(\eta)} \subset \R^d\times \R^d$ and by ${\rm Id}_\cL$ the identity function on $\cL$.
Assume that:
\begin{itemize}
\item[$\bullet$]  for every $(q, p)\in \cL, \partial_p(\pi_1\circ \Phi)(q,p)$ is invertible, where $\pi_1: \R^d \times \R^d\longrightarrow \R^d$ denotes the projection on the $q$-component;
\item[$\bullet$]  $\Phi_{|\cL}\equiv{\rm Id}_\cL$ .
\end{itemize}
Then,  $\eta$ is as regular as $\Phi$ is.
\end{lemma}

\begin{proof} 
For $q_0\in \R^d$, let us solve $\pi_1\circ \Phi(q, p)=q$ in a neighbourhood of $(q_0, \eta (q_0)) \in \cL$.
     Since  $\partial_p(\pi_1\circ \Phi)(q_0, \eta(q_0))$ is invertible, by the implicit function theorem, there exists a neighbourhood $U\times V$ of $(q_0, \eta (q_0))$, where $U,V\subset \R^d$, and a  function $\nu:U\to V$, which is as regular as $\Phi$,  such that
    $$\forall (q, p)\in U\times V: \qquad \pi_1\circ \Phi(q, p)=q \quad \Longleftrightarrow \quad p=\nu (q).$$ The continuity of $\eta$ and $\nu$ at $q_0$ implies that $\eta=\nu$ in a neighbourhood of $q_0$. Hence $\eta$ is as regular as $\Phi$ is. 
    \end{proof}
    
    \medskip

We can now deduce the following regularity result for Lipschitz Lagrangian $(m,n)$ periodic graphs of  a positive symplectic twist map.

\begin{prop}\label{Plocaltwist}
Let $F: \R^{d}\times \R^d \righttoleftarrow$  be a lift of a positive symplectic twist map  $f: \T^d\times \R^d \righttoleftarrow$.
If $\cL$ is a  {Lipschitz} Lagrangian $(m,n)$-periodic  graph of $F$, then $\cL$ is as regular as $F$ is.
\end{prop}

\begin{proof}
First of all, observe that since $\cL$ is a $(m,n)$-periodic graph, then $F^n_{|\cL}\equiv{\rm Id}_\cL$.\\
Moreover, it follows from Proposition \ref{p(m,n)minimizing} that every  orbit of $F$ starting at a point of  $\cL$ is action-minimizing (we noticed in Remark \ref{twist} (i) that the superlinearity condition for the generating function is satisfied for positive symplectic twist maps). 
This implies that along an orbit starting at $(q,p)\in \cL$,  there are no conjugates points (see Remark \ref{twist} (ii)), {\i.e.}, 
if we denote by $(q_n, p_n):=F^n(q, p)$, we have $\det(\partial_p q_n) \neq0$. This means that $\partial_p(\pi_1\circ F^n)(q,p)$ is invertible for every $(q,p)\in \cL$.
Hence, Lemma \ref{lemreg} implies that $\cL$ is as regular as $F$ is.
\end{proof}

\bigskip

\subsection{Symplectic properties of $(m, n)$-periodic graphs}\label{perlagtori}
All properties that we proved in sections \ref{subsecactionminprop} and \ref{secregularityLaggraphs} are based on the assumptions that the $(m, n)$-periodic graphs are Lagrangian. 
In this section, we provide some sufficient conditions that imply this property.\\

 More specifically, we shall prove the following proposition.

\begin{prop}\label{P(m,n)Lagr}
Let $F:   \R^{d} \times \R^d \righttoleftarrow$ be a lift of a  symplectic twist map  $f: \T^d\times \R^d\righttoleftarrow$.
Assume that $\cL={\rm graph}(\gamma)$ is a {continuous} $(m, n)$-periodic graph of $F$ such that for all $q\in \R^d$
$$\det\Big(\partial_p(\pi_1\circ F^n)(q, \gamma(q))\Big)\neq0\quad{\rm and} \quad \det\Big(DF^n(q, \gamma(q))- X\,{\mathbb I_{2d}}\Big)=(X-1)^{2d},$$
where ${\mathbb I_{2d}}$ denotes the $2d$-dimensional identity matrix and $X$ is a complex variable.\\
Then, {$\cL$ is $C^1$ and} {the two following assertions are equivalent.
\begin{itemize}
\item[{\rm (i)}] $\cL$ is Lagrangian;
\item[{\rm (ii)}] $\Big({DF^n(q, \gamma(q)) -\mathbb I_{2d}} \Big)^2={\mathbb O}_{2d}$, {$\forall q\in\R^d$}.\\
\end{itemize}}
\end{prop}

\medskip
Let us first prove the following Lemma.

\begin{lemma}\label{Leigenspace}
 Under the assumptions of  Proposition \ref{P(m,n)Lagr}, {$\cL$ is $C^1$ and} we have 
 $$T_{(q, \gamma(q))}\cL=\ker (DF^n{(q, \gamma(q))}-{\mathbb I}_{2d}) \qquad \forall\; q\in \R^d$$ 
 where $T_{(q, \gamma(q))}\cL$ denotes the tangent space to $\cL$ at the point $(q, \gamma(q))$.
\end{lemma}

\begin{proof}[Proof of Lemma \ref{Leigenspace}]
Let $x:=(q,\gamma(q)) \in \cL$. Since $\det\Big(\partial_p(\pi_1\circ F^n)(q, \gamma(q))\Big)\neq0$ for every $q\in \R^d$, 
by applying Lemma \ref{lemreg}, we deduce that $\cL$ is at least $C^1$.
Since $\cL$ is $(m, n)$-periodic, then $DF^n_{|T_x\cL}={\rm Id}_{T_x\cL}$ and therefore $T_x\cL\subseteq \ker (DF^n(x)-{\mathbb I}_{2d})$.
Moreover, observing that $\dim T_x\cL=d$, in order to prove the claim, it is enough to show  that $\dim \ker (DF^n-{\mathbb I}_{2d})\leq d$. \\
Assume by contradiction that $\dim \ker (DF^n(x)-{\mathbb I}_{2d})\geq d+1$. Denote $V(x):=\ker D\pi_1(x)$, using the fact that $\dim V(x)=d$, we can deduce that the intersection $V(x)\cap \ker (DF^n(x)-{\mathbb I}_{2d})$ contains at least one non-zero vector that we denote by $\begin{pmatrix} 0\\ v\end{pmatrix}$. Since $\det\Big(\partial_p(\pi_1\circ F^n)(x)\Big)\neq0$, we have 
$$d\pi_1\circ DF^n(x)\begin{pmatrix} 0\\ v\end{pmatrix}=\partial_p(\pi_1\circ F^n)(x)v\neq0$$
and therefore $DF^n(x)\begin{pmatrix} 0\\ v\end{pmatrix}\neq\begin{pmatrix} 0\\ v\end{pmatrix}$, which contradicts the fact that $$\begin{pmatrix} 0\\ v\end{pmatrix}\in \ker(DF^n(x)-{\mathbb I}_{2d}).$$
\end{proof}

\medskip
Let us now prove Proposition \ref{P(m,n)Lagr}.\\

\begin{proof}[Proof of Proposition \ref{P(m,n)Lagr}]
{The fact that $\cL$ is $C^1$ follows from Lemma \ref{Leigenspace}. Let us now prove the equivalence between (i) and (ii). {Hereafter, we denote by $V(x):=\ker D\pi_1(x)$ the vertical space at $x\in \cL$}. }
\begin{itemize}
\item[$\bullet$] {We assume (ii), {\it i.e.,} that 
\begin{equation}\label{newcondition}
\Big({DF^n(x)- \mathbb I_{2d}}\Big)^2={\mathbb O}_{2d} \quad {\forall\; x\in \cL}
\end{equation} and we prove that $\cL$ is Lagrangian.} Equation \eqref{newcondition} implies that the minimal polynomial of $DF^n(x)$ is a divisor of $(X-1)^2$, hence  $DF^n(x)$
 has no other eigenvalue than $1$.
It follows from Lemma \ref{Leigenspace} that 
 for every $x\in \cL$, the eigenspace of $DF^n(x)$  corresponding to the eigenvalue $1$  is $T_x\cL$. Therefore $DF^n(x)$ is not diagonalizable and its minimal polynomial is $(X-1)^2$.
 Choosing a basis (not necessarily symplectic) in which the first $d$  vectors form a basis of  $T_x\cL$, the matrix of $DF^n(x)$ becomes 
$${\mathbb M}:=\begin{pmatrix}{\mathbb I}_d & {\mathbb D}(x)\\ {\mathbb O}_d & {\mathbb I}_d
\end{pmatrix}$$
where ${\mathbb O}_d$ is the $d$-dimensional zero square matrix,  ${\mathbb I}_d$ the d-dimensional identity matrix and
{${\mathbb D}(x)$} is a $d\times d$ invertible matrix. Since the vertical  $V(x)$ is transverse to $T_x\cL$, there exists a $d$-dimensional square matrix ${\mathbb H}(x)$ such that in these coordinates
$$
V(x) =  \{(\mathbb H(x)\, y,y):\; y\in\R^d\}.$$ 
The image by ${\mathbb M}^k$ of a vector $\begin{pmatrix} \mathbb H(x)\, y\\ y\end{pmatrix}$ of $V(x)$ is then 
$\begin{pmatrix} \mathbb H(x)\,y+k \, \mathbb D(x)\, y\\ y\end{pmatrix}$.   For $k$ large enough, ${\mathbb D}(x)+\frac{1}{k}{\mathbb H}(x)$ is invertible and in coordinates 
$$DF^{nk}V(x)=\{ (x, \frac{1}{k}\left( {\mathbb D}+\frac{1}{k}\mathbb{H}\right)^{-1}x); x\in\R^d\}$$
converges as 
$k$ tends to $\infty$ to $T_x\cL$ in the Grasmannian of $d$-dimensional subspaces of {$T_x(\R^d\times \R^d)$}. \\
Being the limit of a sequence of Lagrangian subspaces, then $T_x\cL$ is Lagrangian. 
\item[$\bullet$]
{Now, we assume that $\cL$ is Lagrangian. Let us fix {$x \in \cL$}. Then we have $V(x)\oplus T_x\cL=\R^d\times \R^d$. Then for every $(v, w)\in V(x)\oplus T_x\cL$, a result of   Lemma \ref{Leigenspace} is that $DF^n(x) w = w$.
We deduce
\begin{eqnarray*}
\omega (w, (DF^n(x)-{\mathbb I}_{2d})v) &=& \omega (w, DF^n(x)v) - \omega (w, v)\\
&=&  \omega (DF^n(x)w, DF^n(x)v)-\omega(w, v)=0,
\end{eqnarray*}
where in the last equality we have used that $DF^n(x)$ is a symplectic matrix.
Hence $(DF^n(x)-{\mathbb I}_{2d})v$ is $\omega$-orthogonal to the Lagrangian space $T_x\cL$. This implies that $(DF^n(x)-{\mathbb I}_{2d})v\in T_x\cL=\ker(DF^n(x)-{\mathbb I}_{2d})$ and therefore   $V\subseteq \ker(DF^n(x)-{\mathbb I}_{2d})^2$. 
{Since also $\ker(DF^n(x)-{\mathbb I}_{2d}) \subseteq \ker(DF^n(x)-{\mathbb I}_{2d})^2$, then 
$$
\ker(DF^n(x)-{\mathbb I}_{2d})^2 \supseteq V(x)\oplus T_x\cL=\R^d\times \R^d
$$
and therefore $\big(DF^n(x)-{\mathbb I}_{2d}\big)^2 = {\mathbb O}_{2d}$, as claimed.
}}
 \end{itemize}

\end{proof}



\section{Proof of Theorem \ref{teo 1}}\label{section theo 1}

Before proving Theorem \ref{teo 1} (see section \ref{sstheo1}),  we discuss several results concerning the set of $(m,n)$-periodic  or radially transformed points (see definition below, Remark \ref{radiallytrans}) of a family of symplectic twist maps obtained by a symplectic deformation by a potential (see Notations \ref{notadeformation}).\\

\subsection{Sets of $(m,n)$-periodic and radially transformed points of symplectic twist maps}

Let $f: \T^d \times \R^d \righttoleftarrow$ be symplectic twist map and
let $F: \R^d \times \R^d \righttoleftarrow$ be a lift  of $f$ and $S: \R^d\times \R^d \longrightarrow \R$ its generating function.\\
We consider the symplectic deformation of $f$ by a potential $G\in C^2(\T^d)$ (see Notations \ref{notadeformation}), given by $f_\eps:  \T^d \times \R^d \righttoleftarrow$, with $\eps\in \R$, whose generating functions are
$
S_\eps (q,Q) := S(q,Q) + \eps G(q).
$
We denote by $F_\eps$ a continuous family {of lifts of $f_\eps$}.\\

\begin{notas}
In the case in which $(f_\eps)_\eps$ satisfy the analyticity property ({i.e., when $f$ satisfies the analyticity property, see Definition \ref{Danaly}, and $G$ admits a analytic extension to $\C^d$}), it will be useful to consider the complex extension of this family of maps,  namely $(f_\eps)_{\eps\in \C}$ and $(F_\eps)_{\eps\in \C}$. 
\end{notas}

Let us introduce the following sets  for a fixed $\eps\in\K$, where $\K=\R$ or $\C$.\\
We denote by $\pi_1: \R^d\times \R^d \longrightarrow \R^d$ the projection on the $q$-component, {\it i.e.}, $\pi_1(q,p)=q$.

\begin{defn}\label{sets1}
We introduce the following sets of points and parameters for the family of maps $(F_\eps)_{\eps\in \K}$:
{\small
\begin{eqnarray*}
\cR^*_{(m, n)}(\K) &:=& \{ (\varepsilon, q, p)\in \K^{2d+1}:\; 
\pi_1\circ F^n_\eps(q, p)=q+m,\quad
\det\big(\partial_p(\pi_1\circ F_\eps^n(q, p))\big)\neq0\}\\
\cP_{(m, n)}(\K)&:=& \{ (\varepsilon, q, p)\in \K^{2d+1}:\;  F^n_\eps(q, p)=(q+m,p)\}\\
\cP^*_{(m, n)}(\K)&:=&\cP_{(m, n)}(\K)\cap \cR^*_{(m, n)}(\K).\\
\end{eqnarray*}
}
\end{defn}

\begin{rmk}\label{radiallytrans}
{\bf (i)} A point $(q, p)\in  \K^d\times\K^d$  such that $\pi_1\circ F^n_\eps(q, p)=q+m$  is said to be $(m,n)-${\it radially transformed} by $F_\eps$. Hence,  the set 
$\cR^*_{(m, n)}(\K)$ consists of the points that are $(m,n)-${\it radially transformed} by $F_\eps$ and on which the map $F^n_\eps$ satisfies a non-degeneracy condition, namely 
$\partial_p(\pi_1\circ F_\eps^n(q, p))$ is invertible.\\
 {\bf (ii)} Similarly, $\cP_{(m, n)}(\K)$ consists of the $(m,n)-$ periodic points of $F_\eps$, and $\cP^*_{(m, n)}(\K)$ of non-degenerate $(m,n)$-periodic points  (non-deneracy is meant in the same way as in item (i)).
\end{rmk}

\begin{defn}\label{sets2}
We introduce the following sets of (real) parameters for the family of maps $(F_\eps)_{\eps\in \R}$:
{\small
\begin{eqnarray*}
\cJ^*_{(m, n)}(\R) &:=& \{  \eps\in  \R:\; \exists \; \gamma_\eps: \R^d\to \R^d\, \mbox{Lipschitz and $\Z^d$-periodic, s.t.}\\
&& \qquad \qquad 
\{\eps\}\times {\rm graph}(\gamma_\eps) \subseteq\cR_{(m, n)}^{*}(\R) 
\}\\
\cI_{(m, n)}(\R) &:=& \{  \eps\in  \R:\; \exists \; \gamma_\eps: \R^d\to \R^d\,  \mbox{Lipschitz and $\Z^d$-periodic, s.t.
${\rm graph}(\gamma_\eps)$ is}\\
&& \qquad \qquad \mbox{Lagrangian, invariant by $F_\eps$, and}\;
\{\eps\}\times {\rm graph}(\gamma_\eps) \subseteq\cP_{(m, n)}(\R)
\}\\
\cI^*_{(m, n)}(\R) &:=& \{  \eps\in  \R:\; \exists \; \gamma_\eps: \R^d\to \R^d\,  \mbox{Lipschitz and $\Z^d$-periodic, s.t.
${\rm graph}(\gamma_\eps)$ is}\\
&& \qquad \qquad \mbox{Lagrangian, invariant by $F_\eps$, and}\;
\{\eps\}\times {\rm graph}(\gamma_\eps) \subseteq\cP^*_{(m, n)}(\R)
\}.\\
\end{eqnarray*}
}
\end{defn}

\begin{lemma}\label{Pdiverse}
Assume that $f$ is strongly positive.
The sets of Definitions \ref{sets1} and \ref{sets2} satisfy the following properties.
\begin{enumerate}[{\rm (i)}] 
    \item \label{Thept1}   $\cR^*_{(m, n)}(\K)$ is a $(d+1)$-dimensional  submanifold of $\K\times\K^d\times\K^d$, which is as regular as $F$ is (e.g.,  $C^k$, $k\geq 1$ or $\infty$, or analytic) and it locally coincides with the graph of a function $\Gamma_{m, n}:\cV_{(m, n)}\subset \K\times\K^d\rightarrow \K^d$ defined for $(\eps,q) $ in some open subset $\cV_{(m, n)}$ of $\K\times \K^d$.
     \item\label{Thept2}   $\cJ^*_{(m, n)}(\R)$ is an open subset of $\R$.
    \item\label{Thept3}    $\cI_{(m, n)}(\R)\equiv \cI^*_{(m, n)}(\R)$.
    \item\label{Thept4}  $\cI^*_{(m, n)}(\R)$ is closed and for every  $\eps\in \cI^*_{(m, n)}(\R)$, {$F_\eps$} has exactly one  Lagrangian $(m, n)$-completely periodic graph, denoted ${\rm graph} (\gamma_\eps)$. 
    Moreover, the map $(\eps, q)\in \cI^*_{(m, n)}(\R)\times {\R^d}\mapsto \gamma_\eps(q)$ is as regular as the map $(\eps, q, p)\mapsto F_\eps(q, p)$ (in Whitney's sense, {see \cite{Abraham-Robbin}}), and $\Z^d$-periodic in the $q$-variable.
\end{enumerate}
\end{lemma}

\begin{proof}
\begin{enumerate}[(i)]
    \item 
    The set  $$\cU_{(m, n)}(\K):=\{ (\eps, q, p)\in\K\times \K^d\times\K^d:\;   {\det\big(\partial_p(\pi_1\circ  F_\eps^n(q, p)\big)\neq 0}\}$$ is an  open subset of  $\K\times \K^d\times\K^d$.   At every point $(\eps, q, p)\in \cU_{(m, n)}(\K) $, the map $(\eps, q, p)\in\K\times\K^d\times\K^d\mapsto \pi_1\circ F^n_\eps(q,p)-q-m $ is a submersion, {whose partial derivative with respect to $p$ is invertible}. 
    This implies, by the implicit function theorem, that $$\cR^*_{(m, n)}(\K)=\{ (\varepsilon, q, p)\in \cU_{(m, n)}(\K):\, \pi_1\circ F^n_\eps(q,p)=q+m \}$$ is a $(d+1)$-dimensional  submanifold of $\K\times \K^d\times\K^d$ that  locally coincides with the graph of a function
    $$\Gamma_{(m, n)}: {(\eps, q)\in\cV_{(m, n)}\subset\K\times\K^d\mapsto p=\Gamma_{(m, n)}(\eps, q)\in\K^d}$$ defined on an open subset $\cV_{(m, n)}$ of $\K\times \K^d$, which is as regular as 
{the map     $(\eps, q, p)\mapsto F_\eps(q, p)$ is.}
    Observe that $\Gamma_{(m, n)}$ is  $\Z^d$-periodic in the variable $q$.\\

    \item We assume that $\tilde{\eps}\in\cJ^*_{(m, n)}(\R)$. Then, by definition, there exists a  {Lipschitz} $\Z^d$-periodic function $\gamma_{\tilde{\eps}}: \R^d\rightarrow \R^d$ such that for all $q\in \R^d$
    $$ \pi_1\circ F_{\tilde{\eps}}^n(q, \gamma_{\tilde \eps}(q))=q+m \quad {\rm and} \quad \det\Big(\partial_p\big( \pi_1\circ F_{\tilde{\eps}}(q, \gamma_{\tilde{\eps}}(q))\big)\Big)\neq0.$$
    This implies that $\{\tilde{\eps}\}\times {\rm graph}(\gamma_{\tilde{\eps}})$ is a subset of $\cR^*_{(m, n)}(\R)$. We know  that $\cR^*_{(m, n)}(\R)$ is a $(d+1)$-dimensional submanifold which is locally a graph of a function ${\Gamma}_{(m, n)}$ and $\Z^d$-periodic in the variable $q\in\R^d$ by item \ref{Thept1}. Hence,  there exists  $0<\sigma<\tilde{\eps}$ such that 
    $\{\tilde{\eps}\}\times {\rm graph}(\gamma_{\tilde{\eps}})$ is contained in  ${\Gamma}_{(m, n)}\big( (\tilde{\eps}-\sigma, \tilde{\eps}+\sigma)\times \R^d \big)$. 
     Then, for every $\eps\in (\tilde{\eps}-\sigma, \tilde{\eps}+\sigma)$, the function  $\gamma_\eps:=\Gamma_{(m, n)}(\eps, .)$ is {Lipschitz}, $\Z^d$-periodic  and such that for every $q\in\R^d$
   $$\pi_1\circ F^n_\eps(q,\gamma_\eps(q))=q+m
   \quad{\rm and}\quad \det\big(\partial_p(\pi_1\circ F_\eps^n(q, \gamma_\eps(q))\big)\neq 0.$$
 \noindent 
 This exactly means that $(\tilde{\eps}-\sigma, \tilde{\eps}+\sigma)\subset \cJ^*_{(m, n)}(\R)$, proving that $\cJ^*_{(m, n)}(\R)$ is open.\\

    \item By definition $\cI^*_{(m, n)}(\R)\subseteq\cI_{(m, n)}(\R)$. Let us now consider $\tilde{\eps}\in \cI_{(m, n)}(\R)$. Then $F_{\tilde{\eps}}$ has an $(m,n)$-periodic {Lipschitz} Lagrangian invariant  graph, that we denote by $\cL_{\tilde{\eps}}:={\rm graph}(\gamma_{\tilde{\eps}})$. Proposition \ref{p(m,n)minimizing} and Remark \ref{twist} (ii)
imply that 
    $$ \det\Big(\partial_p\big( \pi_1\circ F_{\tilde{\eps}}^n)(q, \gamma_{\tilde{\eps}}(q)\big)\Big)\neq0 \qquad \forall q\in \R^d,$$
since there are no conjugate points for the  orbits on $\cL_{\tilde{\eps}}$.
This means that ${\tilde{\eps}}\in \cI^*_{(m, n)}(\R)$.\\

\item 
 Let us consider  a sequence $(\eps_k)_{k\geq 1}$ in ${\mathcal \cI^*_{(m,n)}(\R)} = {\mathcal \cI_{(m,n)}(\R)}$ that converges to $\eps_\infty$.
 All the $\eps_k$ are contained in a compact subset   of $\R$ and bounded by some constant $K>0$. \\
{Since $f$ is  positive, 
$S(q,Q)$  satisfies the following superlinearity condition (see Remark \ref{twist} (i) and  \cite[Lemma 27.2]{Gole}): 
 $$
 \lim_{\| Q-q\|\rightarrow +\infty} \frac{S(q, Q)}{\| Q-q\|}=+\infty.
 $$
and therefore there is a compact subset $\cK\subset (\R^{2d})^n$ such that
{\small $$
\sum_{j=0}^{n-1} S(q_j,q_{j+1}) > 2K n \|G\|_\infty + {\max_{q\in \R^d}\sum_{j=0}^{n-1} S\left(q+j\frac{m}{n}, q +(j+1) \frac{m}{n}\right)} \qquad \forall (\Delta q_0, \dots,  \Delta q_{n-1})\notin \cK,
 $$}
where {$q_0:=q$ and $q_{j+1} := q_j+\Delta q_j$} for $j=0, \ldots, n-1$.\\
 We deduce that for $k\in\N$  and $(\Delta q_1, \dots, \Delta q_n)\notin \cK$ we have
{\small
\begin{eqnarray*}
\sum_{j=0}^{n-1} \left( S(q_j,q_{j+1})+\eps_k G(q_j)\right)
&>&  2Kn\,\| G\|_\infty +   {\max_{\|q\|_\infty \leq 1}\sum_{j=0}^{n-1} S\left(q+j\frac{m}{n}, q +(j+1) \frac{m}{n}\right)} +\eps_k \sum_{j=1}^n  G(q_j) \\
&\geq&  {\max_{q\in \R^d}\sum_{j=0}^{n-1} S\left(q+j\frac{m}{n}, q +(j+1) \frac{m}{n}\right)} + n \big(2K - \eps_k)\,\| G\|_\infty\\
&\geq&  {\max_{q\in \R^d}\sum_{j=0}^{n-1} S\left(q+j\frac{m}{n}, q +(j+1) \frac{m}{n}\right)} + n K \| G\|_\infty\\
&\geq & {\max_{q\in \R^d} \left( \sum_{j=0}^{n-1} \left(S\left(q+j\frac{m}{n}, q +(j+1) \frac{m}{n}\right) +\eps_k G\big(q+j\frac{m}{ n}\big)\right)\right).}
 \end{eqnarray*}
 }}
Hence every minimizing
$(m,n)$-periodic orbit  {starting at $(q,p)\in\R^d \times \R^d$, is such that $p$ belongs to a fixed compact subset 
 $ {\Pi^{(m,n)}_{K}}\subset \R^d$, independent of $\eps \in [-K,K]$.
Using Proposition \ref{p(m,n)minimizing}, we deduce that 
$(m, n)$-periodic graphs of $F_{\eps}$, with $\eps\in [-K,K]$, are contained in $\R^d\times {\Pi^{(m,n)}_{K}}$.}

Proposition \ref{Plocaltwist} implies that these graphs are $C^1$ and, {using that they are contained in $\R^d\times {\Pi^{(m,n)}_{K}}$, estimate \eqref{ELipschinequ} from Proposition \ref{lemmac1} implies that they are uniformly Lipschitz.} \\
By Arzelà-Ascoli theorem {(equiboundedness and equicontinuity follow from what remarked above)}, 
there exists a subsequence of $(\eps_{k_j})_j$ such that the $(m, n)$-periodic invariant graphs {of $F_{\eps_{k_j}}$} converge to a Lipschitz graph, which  is then an invariant $(m, n)$-periodic {graph of $F_{\eps_\infty}$} and Lagrangian,  being a uniform limit of {Lipschitz} Lagrangian graphs; observe that this limit graph is also $\Z^d$-periodic, being the uniform limit of $\Z^d$-periodic graphs. 
Again, from Proposition \ref{p(m,n)minimizing} and Remark \ref{twist} (ii) we deduce that all orbits starting at points of ${\rm graph}(\gamma_{\eps_{\infty}})$ are action-minimizing and 
$$ \det\Big(\partial_p\big(\pi_1\circ F^n(q, \gamma_{\eps_\infty}(q))\big)\Big)\neq 0 \qquad \forall q\in\R^d.$$
Hence, $\{\eps_\infty\}\times\{ (q, \gamma_{\eps_{\infty}}(q)):\, q\in\R^d\}  \subseteq \cR^*_{(m, n)}(\R)$. 
{We conclude that ${\eps_\infty} \in \cI^*_{(m,n)}(\R)$, which implies that $\cI^*_{(m,n)}(\R)$ is closed.}

{
Finally, we deduce from item \ref{Thept1} that the map $(\eps, q)\in \cI^*_{(m, n)}(\R)\times {\R^d}\mapsto \gamma_\eps(q)$ is as regular as  $(\eps, q, p)\mapsto F_\eps(q, p)$ is,  {in  Whitney's sense (it coincides, in fact, with the restriction of the map $\Gamma_{m,n}$ on this set)}, {and $\Z^d$-periodic in the $q$ variable}.  
Moreover, it follows from Proposition \ref{uniqtorus} (ii) that for every $\eps\in \cI^*_{(m, n)}(\R)$, {$F_\eps$} has exactly one  invariant Lagrangian $(m, n)$ periodic graph, that must then coincide with ${\rm graph} (\gamma_\eps)$. \\
}

    \end{enumerate}
\end{proof}

\subsection{Proof of Theorem \ref{teo 1}}\label{sstheo1}
This section is organized as follows:

\begin{itemize}
\item[$\bullet$] We shall first prove that the set of $\eps\in \R$ for which $F_\eps$ admits a Lipschitz Lagrangian $(m,n)$-periodic graph is either the {\it whole $\R$} or it has {\it empty interior} (see Lemma \ref{stoca}).
A key tool in the proof is provided by the identity theorem for holomorphic functions.\\
\item[$\bullet$]  We then improve the previous result and prove Theorem \ref{teo 1}. Namely, we show  that the when the  set of $\eps\in \R$ for which $F_\eps$ admits a Lipschitz Lagrangian $(m,n)$-periodic graph is not the whole $\R$, then  it is more than with empty interior: it {\it consists of isolated points}. Moreover, we show that with the additional assumption that $f$ has bounded rate, then this set {\it is at most finite}. \\
The proof strongly relies on the fact that  non-identically zero $1$-dimensional holomorphic functions on a connected set, cannot vanish on sets with an accumulation point. \\
Moreover, under the assumption that $f$ has bounded rate, we show this set is bounded (see Corollary \ref{Pnoinvariantgraph}), from which one deduces that if the potential is not constant (in which case the corresponding deformation is trivial, hence the set is either the empty set or the whole $\R$),  then this set must be at most finite.\\
\end{itemize}

\medskip

Let us start with the following Lemma.\\

\begin{lemma}\label{stoca}
Under the hypotheses (i)-(iii) of Theorem \ref{teo 1}, 
the set 
\[
 \set{\eps\in \R: \, F_\eps \text{ has a {Lipschitz} Lagrangian $(m,n)$-periodic  graph}}
\]
{is the whole $\R$ or it has} empty interior.
\end{lemma}

\begin{proof}
{We will prove that, under these assumptions,  if $\mathcal I_{(m,n)}(\R)$  has non-empty interior, then it must be the whole $\R$; observe in fact that $\mathcal I_{(m,n)}(\R)$ coincides with the set in the statement of the theorem, see Proposition \ref{invarianttori}. 
}

Recall from Lemma \ref{Pdiverse}  that  $\cI_{n, m}(\R)$ is closed.
  We also denote by $F_\eps$ the analytic extension of $F_\eps$ to $\C^d\times\C^d$. 
 Assume by contradiction that $\cI_{n, m}(\R)$ has a  connected component $A$ that is not a single point.  \\

\noindent {\sc Step 1} ({\it Definition and properties of the map $\Delta$}):
 We proved in Lemma \ref{Pdiverse} item \ref{Thept1}  that 
 $\cR^*_{(m, n)}(\C)$  is a $(d+1)$-dimensional complex submanifold of $\C^{2d+1}$. Observe that as a result of {Lemma \ref{Pdiverse} items \ref{Thept3} and  \ref{Thept4}}, the following inclusion holds:
\begin{equation}
\label{da stendere}
\Gamma := \set{(\eps, q, \gamma_\eps(q))\, : \,  {\eps \in A}, \, q\in\R^d } \subseteq \cR^*_{(m, n)}(\C).
\end{equation}
Observe that $\Gamma$ is connected; we denote by $\cV$ the connected component of $\cR^*_{(m, n)}(\C)$ that contains $\Gamma$ and define  the map

\begin{eqnarray}
\label{Delta}
\Delta:\, \cV &\longrightarrow& \C^d \nonumber\\
 (\eps, q, p) &\longmapsto& \pi_2\circ F^n_\eps(q, p) - p,
\end{eqnarray}
{where $\pi_2$ denotes the projection on the $p$-component.}
The map $\Delta$ is holomorphic and vanishes on $\Gamma$.   As the real dimension of $\Gamma$ is $d+1$ and the complex dimension of $\cV$ is $d+1$, we deduce that all coefficients  in the expansion of $\Delta$ at points of $\Gamma$ are zero and thus $\Delta$ vanishes on the whole $\cV$. \\

\noindent {\sc Step 2}  ({\it Definition and properties of the map $\chi$}):
We also  define \begin{eqnarray}
\label{Polyncaract}
\chi:\, \cV &\longrightarrow& {\C_{2d}[X]\times  M_{2d}(\C)}\nonumber\\
 (\eps, q, p) &\longmapsto& {\Big(\det\big(DF^n_\eps(q, p)- X\,{\mathbb I}_{2d}\big), \big(DF^n_\varepsilon(q,p)-{\mathbb I}_{2d}\big)^{2}\Big)},
\end{eqnarray}
where the {$\C_{2d}[X]$ denotes the set of  complex   polynomials with degree at most $2d$, which is identified with  $\C^{2d+1}$, while $M_{2d}(\C)$ is the set of square $2d$-dimensional matrices} and ${\mathbb I}_{2d}$ is the $2d$-dimensional identity matrix. The map $\chi$ is holomorphic and  constant on $\Gamma$, with value { $\big((X-1)^{2d}, {\mathbb O}_{2d}\big)$}, {where ${\mathbb O}_{2d}$ is the $2d$-dimensional zero square matrix.} Indeed, for ${\eps\in A}$, the graph of $\gamma_\eps$ is {analytic and}  Lagrangian and {$F_\eps^n$ restricted to this graph coincides with the map
$(q,p) \mapsto (q+m, p)$.}
Since $F_\eps^n$ is symplectic, then at every point of ${\rm graph}(\gamma_\eps)$ all the eigenvalues of $DF_\eps^n$ must be equal to $1$. {Moreover, we deduce from  Proposition \ref{P(m,n)Lagr} and Remark \ref{twist} that $(DF_\eps^n-{\mathbb I}_{2d})^2={\mathbb O}_{2d}$.}
{Therefore, $\chi$ must be equal to} {$((X-1)^{2d}, {\mathbb O}_{2d})$} on the whole $\cV$.\\

\noindent {\sc Step 3} ({\it The connected component $A$ must be unbounded}): We now show that $A$  is unbounded, both from above and from below, hence $A\equiv \R$, thus concluding the proof.

{In fact, let us assume that $A$ has a least upper bound $\beta$ (one can argue similarly assuming that it has a greatest lower bound or substituting $G$ with $-G$)}. Since $\cI_{n, m}(\R)$ is closed, we have $\beta\in\cI_{n,m}(\R) = \cI^*_{n,m}(\R) \subseteq \cJ^*_{n,m}(\R)$.
Recall that $\cJ_{(m,n)}^*(\R)$ is open (Lemma \ref{Pdiverse} item \ref{Thept2})  and that it consists of parameters $\eps\in\R$ for which there exists a {Lipschitz} $\Z^d$-periodic {(real) graph, that we denote ${\rm graph}(\eta_\eps)$, such that $\{\eps\}\times {\rm graph}(\eta_\eps)\subset \cR^*_{(m, n)}(\R)$}. Let $I$ be the connected component of $\cJ_{(m,n)}^*(\R)$ that contains $\beta$ and, consequently, it contains an open neighborhood of $\beta$ {and hence it properly contains $A$}.

Let us then consider the connected subset of $\cR^*_{(m, n)}(\C)$ 
$$\bigcup_{\eps\in I}\{\eps\}\times {\rm graph}(\eta_\eps)\subseteq \cV$$
and observe that  $\Delta$ must vanish on it, which means that the graphs of the function $\eta_\eps$ for $\eps\in I$ are $(m, n)$-periodic. Moreover, on each of these graphs, $\chi$ is constantly equal to  {$((X-1)^{2d}, {\mathbb O}_{2d})$} . We deduce from Proposition \ref{P(m,n)Lagr} and Remark \ref{twist} that these graphs are Lagrangian, and then from Proposition \ref{invarianttori} that they are invariant. Therefore, $I\subseteq \cI_{n, m}(\R)$, which contradicts the  fact that $A$ is a connected component of  $\cI_{n, m}(\R)$.
\end{proof}

\bigskip

We can now prove Theorem \ref{teo 1}.\\

\begin{proof}[\textbf{Proof of Theorem \ref{teo 1}}]
As observed before, $\cI_{n, m}(\R)$ coincides with the set in the statement of the theorem, see Proposition \ref{invarianttori}. \\

\noindent {\sc Proof of part {\rm I}:} Let us assume that $\cI_{n, m}(\R)$ has an accumulation point $\tilde{\eps}$, which belongs to $\cI_{n, m}(\R)$ since it is a closed set (see Lemma \ref{Pdiverse} items (iii)-(iv)). \\

\noindent{\sc Step 1} ({\it Definition of the function $\widetilde \Delta$}):
Let us consider
$$\Gamma_{\tilde{\eps}} := \set{(\tilde{\eps}, q, \gamma_{\tilde{\eps}}(q))\, : \, q\in\R^d } \subset \cR^*_{(m, n)}(\C)$$
and let us denote by $\cV$ the connected component of $\cR^*_{(m, n)}(\C)$ that contains $\Gamma_{\tilde{\eps}}$. Since $F_\eps$ is $\Z^d$-periodic in $q$, so it is its holomorphic extension. Hence $\cR^*_{(m, n)}(\C)$ is invariant by the translation $(\eps, q, p)\mapsto (\eps, q+k, p)$ for every $k\in\Z^d$.
In particular, it follows from Lemma \ref{Pdiverse} item (i)  that there exists an open neighboorhood  $\cU$  of $\Gamma_{\tilde{\eps}}$ in $\cV$ that coincides with the graph of an analytic function 
\begin{eqnarray*}
\Gamma_{m,n}: \cO \subset \C\times \C^d &\longrightarrow& {\C^d}\\
(\eps, q) &\longmapsto & \Gamma_{m,n}(\eps, q)
\end{eqnarray*}
where $\cO \subset \C\times \C^d$ is a $\delta$-neighborhood of $\{\tilde{\eps}\} \times \R^d $, for some $\delta>0$, $\Gamma_{m.n}$ is $\Z^d$-periodic in $q$ and $\{\tilde{\eps}\} \times {\rm graph}(\Gamma_{m,n}(\{\tilde{\eps}\}, \cdot)) \equiv \Gamma_{\tilde{\eps}}$. 
\\

Recalling the definition of the function $\Delta$ in \eqref{Delta}, let us define
\begin{eqnarray*}
\widetilde \Delta : \cO\subset \C\times \C^d &\longrightarrow& \C^d\\ 
(\eps, q) &\longmapsto&  \Delta ({\eps, q, \Gamma_{m,n} (\eps, q)}),
\end{eqnarray*} 
which is clearly analytic in $\cO$.\\

\noindent {\sc Step 2} ({\it Identity theorem and vanishing of $\widetilde \Delta$}): Since $\tilde{\eps}$ is an accumulation point of $\cI_{n, m}(\R)$, there exists a sequence of $\{ \eps_n\}_n \subset \cI_{n, m}(\R)$ such that $\eps_n \rightarrow \tilde{\eps}$ and 
the corresponding graphs 
$$
\Gamma_{\eps_n} := \set{({\eps_n}, q, \gamma_{\eps_n}(q))\, : \, q\in\R^d } \subset \cU
$$
(here we use the  uniqueness of these graphs, see Lemma \ref{Pdiverse} (iv)).\\
Observe that for every $q\in \R^d$, the function
$\widetilde{\Delta}(., q)$ is an analytic function of one complex variable, that vanishes on the set $\{\eps_n\}_n \cup \{\tilde{\eps}\}$, which has an accumulation point; hence, it vanishes identically on $B_\delta(\tilde{\eps})\subset \C$.
Since this holds for every $q \in \R^d$, we can conclude that $\widetilde{\Delta}$ vanishes on $B_\delta(\tilde{\eps}) \times \R^d \subset \cO$; observe that this set contains subset of real dimenstion $d+1$ and therefore
$\widetilde{\Delta} \equiv 0$ on $\cO$, which implies that $\Delta\equiv 0$ on $\cU$ and hence on $\cV$. \\

\noindent {\sc Step 3} ({\it Identity theorem applied to the function $\chi$}): 
In the same way, one can show that $\chi$ (see \eqref{Polyncaract}) is identically equal to  {$((X-1)^{2d}, {\mathbb O}_{2d})$} on the whole $\cV.$\\

\noindent {\sc Step 4} ({\it Conclusion of the proof of part I}): 
Therefore,  for every $\eps \in (\tilde{\eps}-\delta, \tilde{\eps}+\delta)$ the graphs of $\Gamma_{m,n}(\eps, \cdot)$, {with $q\in \R^d$}, are $(m, n)$-periodic (as it follows from the vanishing on them of $\Delta$), Lagrangian (as it follows from the fact that
$\chi$ is constantly equal to  {$((X-1)^{2d}, {\mathbb O}_{2d})$} on them, see Proposition \ref{P(m,n)Lagr}) and invariant (as it follows from Proposition  \ref{invarianttori}).\\

Therefore $(\tilde{\eps}-\delta, \tilde{\eps}+\delta) \subset \cI_{n, m}(\R)$; it follows from Lemma \ref{stoca}  that $\cI_{n, m}(\R)\equiv \R$ and this completes the first part of the proof: the set is either $\R$ or it must consists of isolated points. \\

\noindent {\sc Proof of part {\rm II}:} If in addition $f$ has bounded rate and $G$ is not constant, then it follows from Corollary \ref{Pnoinvariantgraph} that $\cI_{n, m}(\R)$ must be bounded; therefore, we deduce that it consists of at most finitely many points (otherwise, these points would have an accumulation point, thus contradicting the property of being isolated).

\end{proof}

\bigskip

\section{Proof of Theorem \ref{teo 2}}\label{section theo 2}

The proof of Theorem \ref{teo 2}  consists of the following steps (we assume the notation and assumptions of Theorem \ref{teo 2}):

\begin{itemize}
\item[$\bullet$] Let $(m,n)\in \Z^d\times \N^*$, with $m$ and $n$ coprime. We show that the existence of infinitely many $\eps\in \R$, accumulating to $0$,  for which $F_\eps$ has an $(m,n)$-completely periodic Lagrangian graph, implies the vanishing of certain
Fourier coefficients of $G$, determined by $(m,n)$ (see Lemma \ref{lemmino base}).
\item[$\bullet$] Using assumption ({\it iv}) of Theorem \ref{teo 2}, we deduce that all but at most finitely many Fourier coefficients of $G$ must vanish (see Proposition \ref{piazzetta} and Lemma \ref{lemmasetvectors}).
\item[$\bullet$] Finally, being $G$ a trigonometric polynomial, the proof of Theorem \ref{teo 2} will follow from Theorem \ref{teo 1}. \\
\end{itemize}

Let us start with proving the following Lemma.\\

 \begin{lemma}\label{lemmino base} 
{Under the notation  and assumptions (i)-(iii) of Theorem 2.} 
 Let $(m,n)\in \Z^d\times \N^*$, with $m$ and $n$ coprime, and assume that there exist
{infinitely many values of $\eps\in \R$, accumulating to $0$,} for which
$F_\eps$ has an $(m,n)$-completely periodic Lagrangian graph.
  Then, for every $\nu \in \Z^d\setminus\{0\}$ such that $\langle \nu , \frac{m}{n} \rangle \in \Z$, we have $\widehat{G}(\nu)=0$  (where $\widehat{G}(\nu)$ denotes the 
   {$\nu$-th} Fourier coefficient of $G$).
\end{lemma}

{
{

\begin{proof}{The result is trivial when $G$ is constant. Hence we assume that $G$ is not constant.}\\

\noindent {\it Step 1 (Preliminaries):}
Recall Lemma \ref{Pdiverse} (and notations therein) and Proposition \ref{invarianttori}. 
{It follows that there exist
some open subset $\cV_{(m, n)}$ of $\R\times \R^d$ containing a $\delta$-neighbourhood of $\{0\}\times\R^d$, and 
 a $C^1$ function
$ (\e, q) \in \cV_{(m, n)}\subset \R\times\R^d\longmapsto \Gamma_{m, n}(\e,q) \in \R^d$, such that $\Gamma_{m, n}$ is $\Z^d$-periodic in $q$,  its graph is in $\cR^*_{(m, n)}(\R)$ and the graph   $\gamma_0:=\Gamma_{m, n}(0, \cdot)$ is  Lagrangian and $(m, n)$-completely periodic for $F$.}\\

Let us denote by $\{\eps_k\}_{k\geq 1}$  the values of $\eps$, accumulating to $0$, whose existence is assumed in the statement. 
We can assume that $\forall k,  |\eps_k|<\delta$. 
{
If we denote by $\gamma_{\eps_k}:\R^d\to\R^d$ the $\Z^d$-periodic maps whose graphs are $(m, n)$-completely periodic and Lagrangian for $F_{\eps_k}$, then -- following the same argument as in the proof of Lemma \ref{Pdiverse} item (iv) --
we can verify the hypotheses to apply Arzel\`a-Ascoli theorem and deduce that,  as $k$ tends to $\infty$, $\gamma_{\eps_k}$ tends to $\gamma_0$ uniformly in $q$ (up to extracting a subsequence). Hence, for sufficiently large $k$, $\gamma_{\eps_k}=\Gamma_{m, n}(\eps_k, \cdot)$.}\\

\noindent {\it Step 2 (Computing the action and its Taylor expansion in $\eps$):}
For $\eps\in (-\delta, \delta)$ and $q\in \R^d$, we denote by $\{q_j^{\eps}\}_{j\in \N}$ the   projection of the orbit of $F_{\eps}$   starting at $q^{\eps}_0=q$ and lying on the graph of $\Gamma_{m, n}(\eps, \cdot)$. Hence, for $k\geq 1$ and $q\in \R^d$,  $\{q_j^{\eps_k}\}_{j\in \N}$  is the   orbit of $F_{\eps_k}$ of rotation vector $\frac{m}{n}$, starting at $q^{\eps_k}_0=q$ and lying on the $(m,n)$- completely periodic invariant Lagrangian graph corresponding to $\gamma_{\eps_k}$.
The  Lagrangian action of $\{q_j^{\eps}\}_{0\leq j\leq n}$ is given by 
\begin{equation}
\label{azione}
{ {\mathcal A}_{(m.n)}^{\eps}(q)} {:=} \sum_{j=0}^{n-1}S_{\eps}({q_{j}^{\eps}, q_{j+1}^{\eps}}) = \sum_{j=0}^{n-1} \left(S(q_{j}^{\eps}, q_{j+1}^{\eps}) + {\eps} G(q_j^{\eps})\right).
\end{equation}
 Since the Lagrangian graph $\gamma_{\eps_k}$ is $(m,n)$-completely periodic,  it follows that  ${\mathcal A}_{(m,n)}^{\eps_k}(q)$ must be constant as a function of $q$ (see Proposition \ref{torusactionminimizing}).
 
 By following a perturbative approach, we are going to get information {on its first-order expansion in power of $\eps$, namely\footnote{In some literature, { ${\mathcal A}^{\eps}_{(m,n)}$ is  referred to
 as the \co{subharmonic potential}, 
 while  its first order term as its (\co{subharmonic}) {\it Melnikov potential} (see, for instance, \cite[p. 1882]{Ramirez-Ros-Pinto:2013}).}}}  
\begin{equation} \label{espansione}
 {  {\mathcal A}^{\eps}_{(m,n)}(q) = {\mathcal A}^{(0)}_{(m,n)}(q) + {\eps} {\mathcal A}^{(1)}_{(m,n)}(q)} + o({\eps}),
\end{equation} 
with the remainder $o(\eps)$ uniform in $q$ (in fact, using $\Z^d$-periodicity, $q$ can be assumed to vary in a compact subset).

\medskip

{
First of all, observe that   we have that   $q_{j}^{\eps} = q_j^0 + O({\eps})$ for every $j=0, \ldots, n$, where $q^0_j := q + j\frac{m}{n}$ and 
the estimate $\| q_{j}^{\eps} - q_j^0\| = O({\eps})$ is uniform in $q$.
This follows from the fact that 
the map 
$$ ({\eps}, q)  \longmapsto q_j^{\eps}=\pi_1\circ F^j(q, \Gamma_{m,n}(\eps, q))$$
 is $C^1$  and $\Z^d$ periodic in $q$.\\

}

Now, expanding ${\mathcal A}^{\eps}_{(m,n)}$ {with respect to} $\eps $ we  get:
\begin{equation}
\label{rem}
\begin{aligned}
 {\mathcal A}^{\eps}_{(m,n)} (q)  &= \sum_{j=0}^{n-1} S(q_{j}^{\eps}, q_{j+1}^{\eps}) + {\eps} G(q_j^{\eps}) \\
 &= \sum_{j=0}^{n-1} S(q^0_{j}, q_{j+1}^0) +  {\eps} \sum_{j=0}^{n-1}  \left \langle (q_{j}^{\eps}-q_{j}^0, q_{j+1}^{\eps}- q_{j+1}^0) ,\, \nabla S(q_{j}^0, q_{j+1}^0 ) \right \rangle \\
 & \qquad +   {\eps} \sum_{j=0}^{n-1}  G(q_j^0) + o({\eps}),
\end{aligned}
\end{equation}
with the remainder $o({\eps})$ uniform in $q$.

Let us now observe that 
\begin{eqnarray*}
&& \sum_{j=0}^{n-1}  \left \langle (q_{j}^{\eps}-q_{j}^0, q_{j+1}^{\eps}- q_{j+1}^0) ,\, \nabla S(q_{j}^0, q_{j+1}^0) \right \rangle  \\
&& \quad = \quad
\sum_{j=0}^{n-1}  (q_{j}^{\eps}-q_{j}^0) \partial_q S(q_{j}^0, q_{j+1}^0)  +  (q_{j+1}^{\eps}- q_{j+1}^0) \partial_Q S(q_{j}^0, q_{j+1}^0) \\
&& \quad = \quad
\sum_{j=0}^{n-1}  - (q_{j}^{\eps}-q_{j}^0) \gamma_0(q_j^0)  +  (q_{j+1}^{\eps}- q_{j+1}^0) \gamma_0(q_{j+1}^0) \\
&& \quad = \quad
- (q_{0}^{\eps}- q_{0}^0) \gamma_0(q_{0}^0) + (q_{n}^{\eps}- q_{n}^0) \gamma_0(q_{n}^0) =0,
\end{eqnarray*}
where in the second-last equality we have used that it is a telescopic sum, while in the last equality that
$ q_{0}^{\eps} = q_{0}^0 =q$ and $ q_{n}^{\eps} = q_{n}^0 =q+m$.

{Hence, we get}
\begin{equation}\label{L finale}
{\mathcal A}^{\eps}_{(m,n)} (q)  =   \sum_{j=0}^{n-1} S(q^0_{j}, q_{j+1}^0) +  {\eps} \sum_{j=0}^{n-1}  G(q_j^\0) + o({\eps}).
\end{equation}

\noindent {\it Step 3 (Annihilation of certain Fourier coefficients of $G$):}
By identifying terms in \eqref{L finale} with those in equation \eqref{espansione} we conclude:
{
\begin{eqnarray*}
{\mathcal A}^{(0)}_{(m,n)}(q) &:=& \sum_{j=0}^{n-1} S(q^0_{j}, q_{j+1}^0) = \sum_{j=0}^{n-1} S(q+ j\frac{m}{n}, q + (j+1)\frac{m}{n})\\
&=& \sum_{j=0}^{n-1} h(\frac{m}{n}) = n h(\frac{m}{n})
\end{eqnarray*}
and
$$
{\mathcal A}^{(1)}_{(m,n)}(q):=\sum_{j=0}^{n-1}G(q+j\frac{m}{n}).$$
}

{Since  ${\mathcal A}^{(0)}_{(m,n)} (q)$ is constant,  we conclude that in order to have ${\mathcal A}^{\eps_k}_{(m,n)}(q)$  constant for $k\geq 1$ (as it follows from Proposition \ref{torusactionminimizing}), we necessarily need 
$$ 
{\mathcal A}^{(1)}_{(m,n)}(q)= \sum_{j=0}^{n-1}G(q+ j\frac{m}{n}) \equiv {\rm const.}$$
}
In particular, let $\nu \in \Z^d\setminus\{0\}$ such that $\nu\cdot \frac{m}{n} \in \Z$ and multiply the above relation by $e^{-2\pi i \langle \nu, q\rangle}$; integrating and changing variables, we get:
{
\begin{equation}
\label{coez}
\begin{aligned}
0 &= \int_{\T^d} \sum_{j=0}^{n-1} G(q + j\frac{m}{n}) {e^{ -  2\pi  i \, \langle \nu , q \rangle}}\,dq = 
\sum_{j=0}^{n-1}  \int_{\T^d}  G(q + j\frac{m}{n}) {e^{ - 2\pi i \, \langle \nu , q \rangle}}\,dq  \\
&= \sum_{j=0}^{n-1} \int_{\T^d} G(u) {e^{-  2\pi  i\, \langle \nu, u - j\frac{m}{n}\rangle}}\,du 
=  \sum_{j=0}^{n-1} \int_{\T^d} G(u) e^{-  2\pi i \, \langle  \nu, u \rangle }\,du 
= n\, \widehat{G}(\nu),
\end{aligned}
\end{equation}
}
where $\widehat{G}(\nu)$ denotes  the $\nu$-th Fourier's coefficient of the function $G$. 

\end{proof}
}}

\medskip

We can then prove the following.

\begin{prop} \label{piazzetta} 
Under the assumptions of Theorem \ref{teo 2}, it follows that 
$G$ must be trigonometric polynomial. 
\end{prop}

We need an auxiliary result.

{
\begin{lemma}\label{lemmasetvectors}
Let $q_1,\ldots, q_d \in \Q^d$ be linearly independent vectors over $\R$ and  let $0<a_i<b_i$ for every $i=1, \ldots, d$.
Then, the set
$$
{\mathcal I}:= \{
\nu\in \Z^d \; \mbox{ s.t.}\; \; \langle \nu, \lambda q_i\rangle \not\in \Z  \quad \forall\; \lambda \in (a_i,b_i)\cap \Q \quad
 \forall\; i\in\{1,\ldots, d\} 
\}
$$
is finite.
\end{lemma}

\medskip

{\begin{proof}[\bf Proof of Lemma \ref{lemmasetvectors}.]
For $i=1,\ldots, d$, we denote by $f_i$ the linear form $v\mapsto\langle v, q_i\rangle$. Since  $q_1,\ldots, q_d \in \Q^d$ form a basis, it follows that the map $f:=(f_1, \dots, f_d)$ is an isomorphism of $\R^d$, therefore we can define a norm on $\R^d$, given by
$\| \cdot \|_f:=\max \{ |f_1(\cdot)|, \ldots , |f_d(\cdot)|\}$.
Observe that for every $r>0$, $B_f(0, r):=\{ v\in \R^d:\, \| v\|_f\leq r\}$ is a compact subset of $\R^d$, and therefore $\Z^d\cap B_f(0, r)$ is finite.\\
Let  $\alpha:=\max\{ (b_i-a_i)^{-1}, 1\leq i\leq d\}$; we will  prove that for every $v\in \Z^d\backslash B_f(0, \alpha)$, there exist $i\in \{ 1, \dots, d\}$, $\lambda \in (a_i, b_i)\cap \Q$ such that $f_i(\lambda v)\in \Z$. This implies that 
${\mathcal I} \subset B_f(0,\alpha)\cap \Z^d$, hence proves the thesis.\\
 Let $v\in \Z^d\setminus B_f(0, \alpha)$; since $\| v\|_f>\alpha$, there exists $i\in \{ 1, \dots, d\}$ such that $|f_i(v)|>\alpha$. This implies that $f_i((a_i, b_i)v)$ is an interval of length greater than $\alpha(b_i-a_i)>1$; hence there exists $\lambda\in (a_i, b_i)$ such that $f_i(\lambda v)\in \Z$. We know that $f_i( v)=\langle v, q_i\rangle\in\Q$, since $v\in\Z^d$ and $q_i\in \Q^d$; therefore, $\lambda \in (a_i, b_i)\cap \Q$.

\end{proof}
}

\begin{rmk}
In a similar way, one can show that if we have a family of linearly independent $(q_i)_{1\leq i\leq m} \subset \Q^d$, with $m<d$, then the corresponding set ${\mathcal I}$ (defined as above) would be  the union of a finite number of translated copies of $E^\bot\cap \Z^d$, where 
$E$ is the linear subspace generated by $q_1, \dots, q_m$.  Hence, it is not finite.\\
\end{rmk}

\begin{proof}[{\bf Proof of Proposition \ref{piazzetta}.}]
Up to changing $q_j$ with $-q_j$ and, possibly, restricting to a sub-interval, we can assume that $I_j=(a_j,b_j)$ with $0<a_j<b_j$ for every $j=1,\ldots, d$.

Let ${\mathcal I}$ denote the  set from Lemma \ref{lemmasetvectors} with this choice of $q_1, \ldots, q_d \in \Q^d$  and $a_1, \ldots, a_d, b_1, \ldots, b_d$.

We claim that if $\nu\not \in {\mathcal I}$ then $\widehat{G}(\nu)=0$, where  $\widehat{G}(\nu)$ denotes the 
{$\nu$-th} Fourier coefficient of $G$. Being ${\mathcal I}$ a finite set (see  Lemma \ref{lemmasetvectors}), this allows us to conclude that $G$ is a trigonometric polyonomial:
$$
G(q) = \sum_{\nu \in \mathcal I} \widehat{G}(\nu) e^{2\pi i\, \langle \nu, q \rangle}.
$$

Let us prove the above claim. 
If $\nu \not \in {\mathcal I}$, then there exists $i\in \{1,\ldots, d \}$ and $\lambda \in (a_i,b_i)\cap \Q$ such that
$\langle \nu, \lambda q_i\rangle \in \Z$.  Hence we can apply
Lemma \ref{lemmino base} with $\frac{m}{n}= \lambda q_i $ and conclude that $\widehat{G}(\nu)=0$.
\end{proof}

\bigskip

We can now complete the proof of Theorem \ref{teo 2}.

\begin{proof}[{\bf Proof of Theorem \ref{teo 2}.}]
Since $G$ is  a trigonometric polynomial (see Proposition \ref{piazzetta}), then it admits a holomorphic  extension to $\C^d$.  {Moreover,  since $f$ is completely integrable and strongly positive, then it has also bounded rate (see Remark \ref{rmkciboundedrate})}. Therefore,  the claim follows from Theorem \ref{teo 1}, applied to any choice of $(m,n) \in \Z^d\times \N^*$ such that
$\frac{m}{n}\in \bigcup_{j=1}^d I_jq_j \cap \Q^d$.
\end{proof}

\bigskip

\appendix

\section{Lipschitz inequalities and Green bundles along invariant Lagrangian tori } \label{ALipschitzetGreen}
Here we provide some results for  Lagrangian tori that are invariant by  strongly positive symplectic twist maps, without any prescription on their dynamics.
Namely,  given a strongly positive symplectic twist map:  
\begin{itemize}
\item[$\bullet$]  we  provide a Lipschitz bound for all its $C^1$  Lagrangian invariant tori (Proposition \ref{lemmac1});
\item[$\bullet$] for every symplectic deformation of such a map by a potential, we prove that the Lipschitz bound can be chosen uniformly
for elements of the family corresponding to parameters  belonging to a compact set (Corollary \ref{Pnoinvariantgraph} (i)). Moreover,  we also prove that when the potential is not identically constant, for every parameter large enough, the corresponding element of the family has no invariant tori (Corollary \ref{Pnoinvariantgraph} (ii)). \\
 \end{itemize}

Let us start with this  Proposition.  \\

{\begin{prop}\label{lemmac1}
Let $f:\T^d\times\R^d\righttoleftarrow$ be a strongly positive symplectic twist map. Let  ${F}(q,p)=: (Q(q,p), P(q,p))$ {be a lift} of  $f$  and let $S(q,Q)$ denote a generating function.
\begin{itemize}
\item[{\bf (i)}] Let $\nu: \R^d\longrightarrow \R^d$ be $\Z^d$-periodic and $C^1$, and such that  $\cL := {\rm graph}(\nu)$ is Lagrangian and invariant by $F$.  Then we have:
\begin{equation}\label{Inegdiff}
\partial_q\partial_q S(q, Q(q, \nu(q))) + \partial_Q\partial_Q S(Q^{-1}(q,\nu(q)), q) >0 \qquad \forall q\in \R^d
\end{equation}
{\it i.e.}, it is positive definite as a matrix.\\
\item[{\bf (ii)}] Moreover, the following Lipschitz bound holds for $\cL$:
\begin{equation}\label{ELipschinequ}
\|D \nu(\cdot)\|_{\infty} \leq \max \Big\{ \big\|\partial_q\partial_q S (\cdot, Q(\cdot,\nu(\cdot)) \big\|_{\infty} , \big\|\partial_Q\partial_Q S  (\cdot, Q(\cdot,\nu(\cdot))\big\|_{\infty} \Big\}\,.
\end{equation}
\end{itemize}
\end{prop}
}

\begin{proof}
{{\sc Part} (i):}  We begin with proving positive definiteness, as in \eqref{Inegdiff}.\\

\noindent{\sc Step 1} ({\it Formula for $DF^{-1}$}):
Let us start with a preliminary computation.
Define
$$
\psi(q, p):= (q, Q(q,p)) 
\qquad {\rm and} \qquad \phi(q,Q):= (q, -\partial_q S(q,Q));
$$
it follows from the definition of generating function that $\psi$ and $\phi$ are one the inverse of the other. In particular:
\begin{eqnarray*}
D\psi (q,p) &=&
\left(\begin{matrix}
{\mathbb I}_d & {\mathbb O}_d\\
\partial_q Q(q,p) &\partial_p Q(q,p)
\end{matrix}
\right)
 = (D\phi(\psi(q,p)))^{-1} \\
&=&
\left(\begin{matrix}
{\mathbb I}_d & {\mathbb O}_d\\
- \partial_q\partial_q S & - {\partial_q\partial_Q}{S} 
\end{matrix}
\right)^{-1}_{|(q,Q)=(
\psi(q,p))}  \\
&=& \left(\begin{matrix}
{\mathbb I}_d & {\mathbb O}_d\\
- ({\partial_q\partial_Q}{S})^{-1} \partial_q\partial_q S & - ({\partial_q\partial_Q}{S})^{-1}
\end{matrix}
\right)_{|(q,Q)=(
\psi(q,p))}\,.
\end{eqnarray*}

It follows from this that
\begin{eqnarray*}
DF(q,p)&=& \left(\begin{matrix}
\partial_q Q(q,p)& \partial_p Q(q,p)\\
\partial_q P(q,p) &\partial_p P(q,p)
\end{matrix}
\right)\\
&=&
\left(\begin{matrix}
- ({\partial_q\partial_Q}{S})^{-1} \partial_q\partial_q S  &  - ({\partial_q\partial_Q}{S})^{-1} \\
{\partial_q\partial_Q}{S}- \partial_Q\partial_QS ({\partial_q\partial_Q}{S})^{-1}  \partial_q\partial_q S
& - \partial_Q\partial_QS ({\partial_q\partial_Q}{S})^{-1}
\end{matrix}
\right)_{|(q,Q)=(
\psi(q,p))}\!\!\!\!\!\!\!\!\!\!\!\!\!\!\!\!\!\!\!\!\!\!\!\!\!\!\!\!\!\!\!.
\end{eqnarray*}
Recall that $DF$ is a symplectic matrix, hence its inverse is given by the expression below (we denote by $^T$ the transposed matrix):
{\begin{eqnarray*}
D F^{-1}(Q,P) &=&  (DF(F(q,p)))^{-1} \\
&=& 
-\left(\begin{matrix}
{\mathbb O}_d& -{\mathbb I}_d \\
{\mathbb I}_d &  {\mathbb O}_d
\end{matrix}
\right) (DF)^T(F(q,p))\left(\begin{matrix}
{\mathbb O}_d& -{\mathbb I}_d \\
{\mathbb I}_d &  {\mathbb O}_d
\end{matrix}
\right) \\
&=&
\left(\begin{matrix}
- (\partial_Q\partial_QS ({\partial_q\partial_Q}{S})^{-1})^T
&  (({\partial_q\partial_Q}{S})^{-1})^T 
\\
-({\partial_q\partial_Q}{S} - \partial_Q\partial_QS ({\partial_q\partial_Q}{S})^{-1}  \partial_q\partial_q S)^T
&
- (({\partial_q\partial_Q}{S})^{-1} \partial_q\partial_q S)^T
\end{matrix}
\right)_{|(q,Q)=(
\psi(F(q,p)))}\!\!\!\!\!\!\!\!\!\!\!\!\!\!\!\!\!\!\!\!\!\!\!\!\!\!\!\!\!\!\!\!\!\!\!\!\!.
\end{eqnarray*}
}

\medskip

\noindent {\sc Step 2} ({\it Order relation of subspaces and proof of \eqref{Inegdiff}}): Let $\cL$ as in the hypotheses. {Let $x:=(q,p) \in \R^d\times \R^d$};
we recall that  there is an order relation on the set of Lagrangian subspaces of $T_x(\R^d\times\R^d)$ that are transverse to the fibers $V(x):=\ker D{\pi_1(x)}$.\\

In \cite[{Section 2.1}]{Arnaud:2013} {(see also \cite[Section 3.1]{Arnaud:2008})}, to every $L_-$, $L_+$ {Lagrangian subspaces} that are transverse to $V(x)$, is associated a quadratic form 
{defined on the quotient linear space}
$$Q(L_-, L_+):T_x(\R^d\times\R^d)/V(x)\rightarrow \R$$
in the following way:
 in the usual coordinates $(\delta q, \delta p)$ of $T_x(\R^d\times\R^d)$, $L_\pm$ is the graph of a symmetric matrix\footnote{The corresponding matrix is symmetric because $L_\pm$ is Lagrangian.} {$S_\pm$} and in the coordinates $\delta q$ the matrix associated to $Q(L_-, L_+)$ is given by $S_+-S_-$.
{This allows one to define an order relation on the set of these Lagrangian subspaces:}
$L_+>L_-$ if and only if  $Q(L_-, L_+)$ is positive definite.

In \cite[Proposition 7]{Arnaud:2013} (see also  \cite{BialyMcKay2004}), it is proven  that if $f$ is a strongly positive symplectic twist map, then at every $x\in \R^d\times\R^d$ whose (lifted) orbit is minimizing, the two Lagrangian subspaces
$$G_1(x):=DF(F^{-1}(x))V(F^{-1}(x)) \quad {\rm and} \quad G_{-1}(x):=DF^{-1}(F(x))V(F(x))$$ are transverse to the vertical $V(x)$ and satisfy
\begin{equation}\label{InegGreen}G_{-1}(x)<G_1(x).\end{equation}
As the orbit of every point that is contained in $\cL$ is minimizing, {see \cite[Theorem 35.2]{Gole}}, this inequality is true for all the points of $\cL$.
{Moreover, for $(q,p)\in \R^d\times \R^d$:
\begin{itemize}
\item the {direct image of the vertical}  $G_1(q,p):=DF(F^{-1}(q,p)))V(F^{-1}(q,p))$ turns out to be the graph of the matrix $\partial_Q\partial_Q S(\psi(F^{-1}(q,p)))$,
\item the {inverse image of the vertical} $G_{-1}(q,p):=DF^{-1}(F(q,p)))V(F(q,p))$ turns out to be the graph of the matrix $-\partial_q\partial_q S(\psi(F(q,p)))$.
\end{itemize}
Hence, the condition  $G_{-1}(q,p)<G_1(q,p)$ reads:
$$
\partial_q\partial_q S(\psi(F(q,p)))  + \partial_Q\partial_Q S(\psi(F^{-1}(q,p)))  >0,
$$}
thus proving \eqref{Inegdiff}.

\smallskip

\noindent {\sc Part {\rm (ii)}:}
We will now prove that $\partial_q\partial_q S(q,Q)$ and $\partial_Q\partial_Q S(q,Q)$ provide some Lipschitz inequalities for  $\cL$, thus proving \eqref{ELipschinequ}.\\ The idea is the following: we will construct a family $(\tilde f_t)_{t\in (0, 1]}$ of symplectic twist maps such that $\tilde f_1=f$ and such that for every $t\in (0, 1]$ the tangent space to $\tilde f_t(\cL)$ is transverse to $G_1$ and $G_{-1}$. Proving that there is an inequality among these three Lagrangian subspaces (or more exactly, among the symmetric matrices that define these subspaces) for $t$ small enough and that these three subspaces are transverse for every $t\in (0, 1]$, we will deduce the same inequality for $t=1$, {\it i.e.}, for $f$.\\
 Since $f$ is strongly positive, there exists a constant $\alpha>0$ such that
$${\partial_q\partial_Q}{S}(v, v)\leq -\alpha\| v\|^2 \qquad \forall q, Q, v \in\R^d.$$

\smallskip
\noindent {\sc Step 1} ({\it Interpolation of symplectic twist maps}):
We use  a 1-parameter family of symplectic twist maps that is built in the proof of \cite[Theorem 41.6]{Gole}. \\
We fix a smooth non-negative and non-increasing function $\eta:(0; +\infty)\to[0, +\infty)$ such that $\eta(1)=\eta'(1/2)=0$, $\eta(1/2)=1$ and $\lim_{t\to 0^+}\eta(t)=+\infty$. We also choose $\eta$ such that the extension of $1/\eta$ at $0$, which is continuous, is differentiable at $0$. Then we consider the 1-parameter family of generating functions $(\widetilde{S}_t)_{t\in (0,1]}$ that are defined by
$$
\widetilde{S}_t(q, Q):=\begin{cases}\frac{\alpha}{2}\, \eta(t)\,\| Q-q\|^2& \text{for}\quad 0<t\leq 1/2\\
\frac{\alpha}{2}\, \eta(t)\, \| Q-q\|^2 + (1-\eta (t))S(q, Q) & \text{for}\quad 1/2\leq t\leq 1.
\end{cases}
$$
For every $t\in (0, 1]$, the function $\widetilde{S}_t$ is the generating function of a symplectic twist map $\tilde{f}_t$ that is strongly positive.  
More precisely, we have $\tilde f_1=f$;  {let $(\widetilde F_t)_t$ denotes a continuous family of lifts of $\tilde f_t$ such that} for $0<t\leq 1/2$, 
$$\widetilde F_t(q, p)= (q+(a\eta(t))^{-1}p, p).$$
We associate  $\widetilde\psi_t$ to $\widetilde F_t$ as $\psi$ was associated to $F$ at the beginning of the proof.\\
We introduce for $0< t\leq 1$
\begin{itemize}
\item the direct image of the vertical  $\widetilde G^t_1(q,p):=D\widetilde F_t(\widetilde F_t^{-1}(q,p))V(\widetilde F_t^{-1}(q,p))$ turns out to be the graph of the matrix $\partial_Q\partial_Q \widetilde S_t(\widetilde\psi_t(\widetilde F_t^{-1}(q,p)))$,
\item the  inverse image of the vertical $\widetilde G_{-1}^t(q,p):=D\widetilde F_t^{-1}(\widetilde F_t(q,p))V(\widetilde F_t(q,p))$ turns out to be the graph of the matrix $-\partial_q\partial_q \widetilde S_t(\widetilde\psi_t(\widetilde F_t(q,p)))$.
\item when $(q, p)\in \widetilde F_t(\cL)$, $T_{(q, p)}\widetilde F_t(\cL)$ is the tangent space to $\widetilde F_t(\cL)$. For $t>0$ small enough, this is  the graph of a symmetric matrix that we denote by $\cT_t(q, p)$.
\end{itemize}
For $t\in (0, 1/2]$, we have 
$$\partial_Q\partial_Q \widetilde S_t(\psi_t(\widetilde F_t^{-1}(q,p)))=(\alpha \,\eta(t))^{-1}{\mathbb I}_d \quad {\rm and }\quad \partial_q\partial_q \widetilde S_t(\widetilde\psi_t(f_t^{-1}(q,p)))=-(\alpha \,\eta(t))^{-1}{\mathbb I}_d.$$ 
Because $\cT_t$ is bounded for $t>0$, sufficiently small, there exists $\eps>0$ such that
$$ -\partial_q\partial_q \widetilde S_t(\widetilde\psi_t(\widetilde F_t^{-1}(q,p)))<\cT_t(q, p)<\partial_Q\partial_Q \widetilde S_t(\widetilde\psi_t(\widetilde F_t^{-1}(q,p)))$$
for all $(q, p)\in \widetilde F_t(\cL)$ and  $t\in (0, \eps]$.\\

\noindent {\sc Step 2} ({\it Extension of the interpolation}): Observe that all these matrices continuously depend on $t$, when they are defined. Observe also that for $t\in (0, 1]$ and $(q, p)\in \widetilde F_t(\cL)$:
\begin{itemize}
    \item[a)] all the subspaces $\widetilde G^t_1(q,p)$ and $\widetilde G^t_{-1}(q,p)$ are transverse to the vertical;
    \item[b)] since $\widetilde G^t_1(q,p)=D\widetilde F_t( \widetilde F_t^{-1}(q,p))V(\widetilde F_t^{-1}(q,p))$ and $$T_{(q, p)} \widetilde F_t(\cL)=D \widetilde F_t(\widetilde F_t^{-1}(q,p))\big(T_{\widetilde F_t^{-1}(q,p))}\cL\big),$$ $\widetilde G^t_1(q,p)$ and $T_{(q, p)}\widetilde F_t(\cL)$ are the image by an isomorphism of two transverse subspaces, therefore they are also transverse;
    \item[c)] the same argument proves that $\widetilde G^t_{-1}(q, p)$ and $T_{(q, p)}\widetilde F_t(\cL)$ are also transverse.\\
\end{itemize}

Let us deduce that for all $t\in(0, 1]$, and all $(q, p)\in \widetilde F_t(\cL)$, $T_{(q, p)}\widetilde F_t(\cL)$ is the graph of a symmetric matrix $\cT_t(q, p)$ such that 
\begin{equation}\label{EncLip}-\partial_q\partial_q \widetilde S_t(\widetilde\psi_t(\widetilde F_t^{-1}(q,p)))<\cT_t(q, p)<\partial_Q\partial_Q \widetilde S_t(\widetilde\psi_t(\widetilde F_t^{-1}(q,p))).\end{equation}
We know that \eqref{EncLip} is true for $t\in (0, \varepsilon]$. Assume that this is not true for some  $t\in (0, 1]$. Let $t_0$ be the infimum of the $t\in (0,1]$ for which there exists $(q, p)\in \widetilde F_t(\cL)$ such that  \eqref{EncLip} does not hold. This means that either the inequality in \eqref{EncLip} does not hold or that $\cT_t(q, p)$ is not defined because $T_{(q, p)}\widetilde F_t(\cL)$ is not transverse to the vertical.\\
Since  the inequality is true for all $t\in (0, t_0)$, this implies,  by continuity in $t$, that for all $(q, p)$ in $\cL$, $T_{\widetilde F_{t_0}(q, p)}\widetilde F_{t_0}(\cL)$ is a graph that satisfies the same inequalities by replacing $<$ by $\leq$. \\ Then, there exists $(q_0, p_0)\in \widetilde F_{t_0}(\cL)$   such that 
$\partial_Q\partial_Q \widetilde S_{t_0}(\widetilde\psi_{t_0}(\widetilde F_{t_0}^{-1}(q,p)))-\cT_{t_0}(q, p)$ is positive semidefinite, but not definite, or that $\cT_{t_0}(q, p)-\partial_q\partial_q \widetilde S_{t_0}(\widetilde\psi_{t_0}(\widetilde F_{t_0}^{-1}(q,p)))$ is positive definite, but not definite. Since the kernel of a positive semidefinite matrix is equal to its isotropic cone, this implies  either that $\widetilde G_1^{t_0}(q, p)$ and $T_{(q, p)}\widetilde F_{t_0}(\cL)$ are not transverse or that $T_{(q, p)}\widetilde F_{t_0}(\cL)$ and $\widetilde G_{-1}^{t_0}(q, p)$ are not transverse, thus providing a contradiction to what observed in items b) and c) above. \\

\noindent {\sc Step 3} ({\it Conclusion}): Now, using  inequality \eqref{EncLip} for $t=1$, we deduce
\begin{equation}
G_{-1}(q, p)<T_{(q, p)}\cL <G_1(q,p) \qquad \forall (q, p)\in\cL,
\end{equation}
which provides Lipschitz inequalities
$$ \|D\nu\|_\infty\leq \max\Big\{ \big\|\partial_q\partial_q S(\psi)\big\|_{\infty, \cL} , \big\|\partial_Q\partial_Q S(\psi)\big\|_{\infty, \cL} \Big\},
$$
where $\|\cdot \|_{\infty, \cL}$ denotes the sup-norm on $\cL$.\\
\end{proof}

Let us now prove the following result, that plays a crucial role in the proof of Theorem \ref{teo 1}, part II.

\begin{cor}\label{Pnoinvariantgraph}
Let $f:\T^d\times\R^d\righttoleftarrow$ be a strongly positive  symplectic twist map { with bounded rate}, let $G:\T^d\rightarrow \R$ be a $C^2$ function and
let   $(f_\eps)_{\eps\in \R}$ be a symplectic deformation of $f$ by $G$ (see Notations \ref{notadeformation}). Then:
\begin{itemize}
    \item[{\rm (i)}] For every $\eps_0\geq0$, there exists $K(\eps_0)\geq 0$ such that every $C^1$-Lagrangian graph invariant by some $f_\eps$, with $|\eps|\leq \eps_0 $, is Lipschitz with Lipschitz constant  $K(\eps_0)$.
    \item[{\rm (ii)}] If $G$ is not constant, there exists $\Lambda>0$ such that for all $\eps\in\R$ such that $|\eps|\geq \Lambda$,  $f_\eps$ does not admit any $C^1$ Lagrangian invariant graph\footnote{This gives another argument for the result of MacKay and Bialy \cite{BialyMcKay2004}, that we cited in {Remark \ref{BialyMcKaySuris}}.}.
\end{itemize}
\end{cor}

\medskip

\begin{rmk}
In dimension 2 and for the standard map, the second point is due to Mather \cite{Mather:1984}  and Aubry-Le Daeron \cite {AubryLD1983}. In any dimension, but only for  generating  functions {$S(q, Q)=h(Q-q)$, with $h$ positive definite quadratic form}, Herman proved the result in \cite{Herman:2018}.
\end{rmk}

\medskip

\begin{proof}
We consider the symplectic deformation of $f$ by a potential $G\in C^2(\T^d)$, given by $f_\eps:  \T^d \times \R^d \righttoleftarrow$, with $\eps\in \R$, whose generating functions are
$
S_\eps (q,Q) := S(q,Q) + \eps G(q).
$
We denote by $F_\eps$ a continuous family {of lifts of $f_\eps$}.\\

\noindent {\sc Part} (i):
Let $\nu:\T^d \longrightarrow \R^d$ be $C^1$ and such that $\cL := {\rm graph}(\nu)$ is Lagrangian and 
invariant by some $f_\eps$, with $|\eps|\leq \eps_0 $. 
Then, considering its lift to $\R^d\times \R^d$, we deduce from \eqref{ELipschinequ} that 
\begin{eqnarray*}
\|D\nu\|_\infty\leq \max\Big\{ \big\|\partial_q\partial_q S\big\|_{\infty} , \big\|\partial_Q\partial_Q S \big\|_{\infty} \Big\}+\eps_0\| G\|_\infty =: K(\eps_0),
\end{eqnarray*}
which is finite since $f$ has bounded rate.\\

\noindent {\sc Part} (ii):
We now assume that $G$ is not constant. Then, there exist $q_1, q_2\in \T^d$ and $v_1, v_2\in \R^d$ such that $D^2G(q_1)(v_1, v_1)>0$ and $D^2G(q_2)(v_2, v_2)<0$. Since $f$ has bounded rate, there exists {$\Lambda>0$} such that
for every $\eps>\Lambda$ and $Q, Q' \in   \R^d$   
$$ 
(\partial_q\partial_qS(q_2, Q) + \partial_Q\partial_QS(Q', q_2) +\eps D^2G(q_2))(v_2, v_2)<0$$
and 
$$(\partial_q\partial_qS(q_1, Q)+\partial_Q\partial_QS(Q', q_1)- \eps D^2G(q_1))(v_1, v_1)<0.$$
Therefore, by \eqref{Inegdiff}, we deduce that $f_\eps$ cannot have an invariant Lagrangian graph when $|\eps|>\Lambda$.

\end{proof}

\color{black}

\end{document}